\documentclass[10pt,twocolumn,twoside]{IEEEtran}
\ifCLASSINFOpdf
  \usepackage[pdftex]{graphicx}
	
\else

\fi
\usepackage{mathrsfs}
\usepackage{amsthm}
\usepackage{amsmath}
\usepackage{amsfonts}
\usepackage{epstopdf}
\usepackage{tabularx}
\usepackage{color}
\usepackage{mathtools}
\usepackage[utf8]{inputenc}
\usepackage[T1]{fontenc}
\usepackage{array}
\usepackage{multirow}
\usepackage{subfigure}
\usepackage[]{algorithm}
\usepackage[]{algpseudocode}
\usepackage[]{pseudocode}
\newtheorem{theorem}{Theorem}[section]
\newtheorem{lemma}[theorem]{Lemma}
\newtheorem{corollary}[theorem]{Corollary}

\DeclareMathOperator*{\argmin}{arg\,min}

\usepackage{color}

\newcommand{\cA}{\mathcal{A}}
\newcommand{\cC}{\mathcal{C}}
\newcommand{\cF}{\mathcal{F}}
\newcommand{\cH}{\mathcal{H}}

\usepackage{algorithm}
\usepackage{algorithmicx}
\newcommand*\Let[2]{\State #1 $\gets$ #2}
\algrenewcommand\algorithmicrequire{\textbf{Input:}}
\algrenewcommand\algorithmicensure{\textbf{Output:}}
\newcommand{\proj}{{\mathrm{proj}}}
\newcommand{\prox}{{\mathrm{prox}}}

\newcommand{\R}{\mathbb{R}}
\usepackage{xcolor}

\begin{document}
\normalsize
%
% paper title
% can use linebreaks \\ within to get better formatting as desired
% \title{Regularized Problems with Nonsmooth Data-Driven Constraints}
\title{Basis Pursuit Denoise with Nonsmooth Constraints}
\author{Robert Baraldi$^1$, Rajiv Kumar$^2$, and Aleksandr Aravkin$^1$.\\
$^1$ Department of Applied Mathematics, University of Washington\\
$^2$ Formerly School of Earth and Atmospheric Sciences, Georgia Institute of Technology, USA; Currently DownUnder GeoSolutions, Perth, Australia}

\maketitle

\normalsize
\begin{abstract}
Level-set optimization formulations with data-driven constraints minimize a regularization functional subject to matching observations
to a given error level. 
These formulations are widely used, particularly for matrix completion and sparsity promotion 
in data interpolation and denoising. 
The misfit level is typically measured in the $\ell_2$ norm, or other smooth metrics.  

In this paper, we present a new flexible algorithmic framework that targets {\it nonsmooth} 
level-set constraints, including  $\ell_1$, $\ell_\infty$, and even $\ell_0$ norms.  
These constraints give greater flexibility for modeling deviations in observation
and denoising, and have significant impact on the solution. 
Measuring error in the $\ell_1$ and $\ell_0$ norms makes the result more robust to large outliers, 
while matching many observations exactly. 

We demonstrate the approach for basis pursuit  denoise (BPDN) problems as well as for 
extensions of BPDN to matrix factorization, with applications to 
interpolation and denoising of 5D seismic data. 
The new methods are particularly promising for seismic applications, where 
the amplitude in the data varies significantly, and measurement noise in low-amplitude 
regions can wreak havoc for standard Gaussian error models.  

%The proposed framework is simpler than the state-of-the-art level-set techniques for residual-constrained formulations, 
%and is not limited to smooth error measures.  
\end{abstract}


\section{Introduction}
Basis Pursuit Denoise (BPDN) seeks a sparse solution to an under-deterimined system of equations that have been corrupted by noise. 
The classic level-set formulation~\cite{van2008probing,aravkin2018level} is given by 
\begin{equation}
\label{eq:basicBPDN}
	\min_x \|x\|_1 \quad \mbox{s.t.} \quad \|\cA(x) - b\|_2\leq \sigma
\end{equation}
where $\cA:\R^{m\times n}\rightarrow \R^d$ is a linear functional taking unknown parameters $x\in\R^{m\times n}$ to observations $b\in\R^{d}$. 
Problem~\eqref{eq:basicBPDN} is also known as a Morozov formulation (in contrast to Ivanov or Tikhonov \cite{oneto2016tikhonov}). 
The functional $\cA$ can include a transformation to another domain, including Wavelets, Fourier, or Curvelet coefficients~\cite{donohue1998bpnd}, 
as well as compositions of these transforms with other linear operators such as restriction in interpolation problems. %~\cite{kumar2013EAGEsind}. 
%Typically, the system is underdetermined with $m\ll n$. The inequality constraint (as opposed to the equality constraint) is due to the fact that with noisy data, fitting the data exactly may be less accurate overall. 
The parameter $\sigma$ controls the error budget, and is based on an estimate of noise level in the data. 

Theoretical recovery guarantees for classes of operators $\cA$ are developed in \cite{candes2006near} and \cite{tropp2006relax}. 
BPDN and the closely related LASSO formulation have applications to compressed sensing \cite{Recht,candes2006near} and 
 machine learning~\cite{girosi1998ml,tibshirani2004least}, as well as to applied domains including MRI \cite{pauly2007mri}. 
Seismic data is a key use case~\cite{fastlowrank, 5dlowrank, splittingschemes},  
where acquisition is prohibitively expensive and interpolation techniques are used to fill in data volumes by promoting parsimonious representations 
 in the Fourier \cite{sacchi} or Curvelet \cite{herrman} domains. Matricization of the data leads to 
low-rank interpolation schemes~\cite{fastlowrank, 5dlowrank, splittingschemes,lowrankstorage}. 
%Interpolation and denoising have proven crucial for accurate inversion and imaging procedures such as Full Waveform Inversion \cite{segoverview}, 
%in particular by removing artifacts and increasing spatial resolution \cite{kumar2013EAGEsind}. 

 While BPDN uses nonsmooth regularizers (including the $\ell_1$ norm, nuclear norm, and elastic net),  
the inequality constraint is ubiquitously smooth, and often taken to be the $\ell_2$ norm as in~\eqref{eq:basicBPDN}. 
Prior work, including~\cite{spgl1, fastlowrank, splittingschemes,aravkin2018level}, exploits the smoothness of the inequality constraint in developing 
algorithms for the problem class.  
Smooth constraints work well when errors are Gaussian, but this assumption fails for seismic data and 
is often violated in general. 

\noindent
{\bf Contributions.}
The main contribution of this paper is to provide a fast, easily adaptable algorithm to solve non-smooth and nonconvex data constraints in general level-set 
formulations including BPDN, 
and illustrate the efficacy of the approach using large-scale interpolation and denoising problems. 
To do this, we extend the universal regularization framework of~\cite{zheng2018sr3} to level-set
 formulations with nonsmooth/nonconvex constraints. 
We develop a convergence theory for the optimization approach, and 
% to different data-driven constraints, and can move fluidly between the transform and temporal/spatial domains, similar to the formulation in . 
illustrate the practical performance of the new formulations for data interpolation and denoising in both sparse recovery and low-rank matrix factorization. 

\noindent
{\bf Roadmap.}
The paper proceeds as follows. 
Section \ref{scn:relax} develops the general relaxation framework and approach. Section \ref{scn:bpdn} specifies this framework to the BPDN setting with nonsmooth, nonconvex 
constraints. 
In Section \ref{scn:bpdn_test} we apply the approach to sparse signal recovery problem and sparse Curvelet reconstruction. 
In Section \ref{scn:lr}, we extend the approach to a low-rank interpolation framework, which embeds matrix factorization within the BPDN constraint. 
In Section \ref{scn:lr_test} we test the low-rank extension using synthetic examples and data extracted from a full 5D dataset simulated on complex SEG/EAGE overthrust model. 
%The model is corrupted with large, sparse noise, thereby facilitating the necessity of an outlier-penalizing constraint. 

%\red{Edit notes: Enough smooth/non-smooth contrasting? Enough sources for other problems? Motivation for nonsmooth constraints done well enough? Roadmap good enough? We need real data, right? }
%

\section{Nonsmooth, nonconvex level-set formulations.}\label{scn:relax}

We consider the following problem class:
\begin{equation}
\label{eq:flippy}
\min_x \phi(\cC(x)) \quad \mbox{s.t.} \quad \psi(\cA(x)-b) \leq \sigma,
\end{equation}
where $\phi$ and $\psi$ may be nonsmooth, nonconvex, but  have well-defined proximity and projection operators:
\begin{equation}
\label{eq:prox}
\begin{aligned}
\prox_{\alpha \phi}(y) &= \argmin_{x} \frac{1}{2\alpha} \| x - y \|^2 + \phi(x)\\
\proj_{\psi(\cdot) \leq \sigma} & = \argmin_{\psi(x) \leq \sigma} \frac{1}{2\alpha} \| x - y \|^2.
\end{aligned}
\end{equation}
Here, $\cC:\mathbb{C}^{m\times n}\rightarrow \R^{c}$ is typically a linear operator that converts $x$ 
to some transform domain, while $\cA:\mathbb{C}^{m\times n}\rightarrow \R^{d}$ is a linear observation operator also acting on $x$. 
In the context of interpolation, $\cA$ is often a restriction operator.

This setting significantly extends that of~\cite{aravkin2018level}, who assume $\psi$ and $\phi$ are convex, $\cC = I$, and use the {\it value function}
\[
v(\tau) = \min_{x}  \psi(\cA(x)-b) \quad \mbox{s.t.} \quad  \phi(x) \leq \tau
\]
to solve~\eqref{eq:flippy} using root-finding to solve 
\(
v(\tau) = \sigma.
\)
Variational properties of $v$ are fully only understood in the convex setting, 
and efficient evaluation of $v(\tau)$ requires $\psi$ to be smooth, 
so that efficient first-order methods are applicable.  

Here, we develop an approach to solve any problem of type~\eqref{eq:flippy}, 
including problems with nonsmooth and nonconvex $\psi, \phi$,
 using only matrix vector products with $\cA, \cA^T, \cC, \cC^T$ and simple nonlinear operators. 
 In special cases, the approach can also use equation solves to gain significant speedup. 
  
% First, define the proximal operator of a function $f(\cdot)$ ($\prox_{\eta f(\cdot)}(y)$) as
% \begin{align*}
% 	\prox_{\eta f(\cdot)}(y) = \argmin_x f(y) + \frac{1}{2\eta}\|x - y\|_2^2
% \end{align*}
%and the projection operator of $F(\cdot)$ to be $\proj_{\sigma F(\cdot)}$
%\begin{align*}
%	\proj_{\sigma F(\cdot)}(y) = \argmin_{x\ \mbox{s.t.}\ F(x)\leq \sigma}\|x - y\|_2^2
%\end{align*}
%where $F(\cdot)$ is usually taken to be some norm-ball. 

\begin{algorithm}[t!]
\caption{Prox-gradient for~\eqref{eq:cost_1}.}
\label{alg:prox-grad}
\begin{algorithmic}[1]
\State{\bfseries Input:} $x^0, w_1^0, w_2^0$
\State{Initialize: $k=0$}
\While{not converged}
\Let{$x^{k+1}$}{$\begin{aligned} x^k - \alpha &\left(\frac{1}{\eta_1} \cC^T(\cC(x)-w_1) \right. \\ &\left.+ \frac{1}{\eta_2} \cA^T(\cA(x)-w_2 -b)\right)\end{aligned} $}
\Let{$w_{1}^{k+1}$}{$\prox_{\frac{\eta_1}{\alpha} \phi}\left(w_1^k - \frac{\alpha}{\eta_1} (w_1^k - \cC(x^{k+1}))\right)$}
\Let{$w_{2}^{k+1}$}{$\proj_{\sigma\mathbb{B}_{\psi}}\left(w_2^k - \frac{\alpha}{\eta_2}(w_2 - (\cA(x^{k+1}) - b)\right)$}
\Let{$k$}{$k+1$}
\EndWhile
\State{\bfseries Output:} $w_1^k, w_2^k, x^k$
\end{algorithmic}
\end{algorithm}

The general approach uses the relaxation formulation proposed in \cite{zheng2018sr3, fastnonsmooth}. 
We use relaxation to split $\phi, \psi$ from the linear map $\cA$ and transformation map $\cC$,  
 %and then augment the cost function such that the distance between the split variable and the dummy variable is minimized. This is a similar concept to the proximal operator in that one is effectively projecting part of the cost function or constraint onto a convex set. 
extending~\eqref{eq:flippy} to 
\begin{equation}
\begin{aligned}
	\min_{x,w_1, w_2} & \phi(w_1) +\frac{1}{2\eta_1} \|\cC(x)-w_1\|^2 + \frac{1}{2\eta_2}\|w_2 - \cA(x)+b\|_2^2 \\
	 \mbox{s.t.}& \quad \psi(w_2)\leq \sigma.
	\label{eq:cost_1}
	\end{aligned}
\end{equation}
with $w_1 \in \mathbb{R}^{c}$ and $w_2 \in \mathbb{R}^{d}$.
In contrast to~\cite{zheng2018sr3}, we use a continuation scheme to force $\eta_i \rightarrow 0$, in order to solve the original formulation~\eqref{eq:flippy}. 
Thus the only external algorithmic parameter the scheme requires is $\sigma$, which controls the error budget for $\psi$.

There are two algorithms readily available to solve~\eqref{eq:cost_1}. The first is prox-gradient descent, detailed in Algorithm~\ref{alg:prox-grad}. 
We let $z = (x, w_1, w_2)$, and define 
\[
\Phi(z) = \phi(w_1) + \delta_{\psi(\cdot) \leq \sigma} (w_2),
\]
where the {\it indicator function} $\delta_{\psi(\cdot) \leq \sigma}$ takes the value $0$ if $\psi(w_2) \leq \sigma$, and infinity otherwise. 
Problem~\eqref{eq:cost_1} can now be written as 
\begin{equation}
\label{eq:lsNC}
\min_z  
\underbrace{\frac{1}{2}\left\| \begin{bmatrix} \frac{1}{\sqrt{\eta_1}}  \cC & - \frac{1}{\sqrt{\eta_1}} I & 0 \\
\frac{1}{\sqrt{\eta_2}} \cA & 0 & -\frac{1}{\sqrt{\eta_2}} I\end{bmatrix} z - \begin{bmatrix} 0\\ b \\ 0\end{bmatrix} \right\|^2}_{f(z)} + \Phi(z).
\end{equation}
Applying the prox-gradient descent iteration with step-size $\alpha$ 
\begin{equation}
\label{eq:pg}
z^{k+1} = \prox_{\alpha \Phi} (z^k - \alpha \nabla f(z^k))
\end{equation}
gives the coordinate updates in Algorithm~\ref{alg:prox-grad}. 

Prox-gradient has been analyzed in the general nonconvex setting by~\cite{attouch2010proximal}. 
However, Problem~\eqref{eq:lsNC} is the sum of a convex quadratic and a nonconvex regularizer. 
The rate of convergence for this problem class can be quantified, and  
\cite[Theorem 2]{zheng2018sr3}, reproduced below, will be very useful here. 

\begin{theorem}[Prox-gradient for Regularized Least Squares]
\label{thm:pg}
Consider the least squares objective 
\[
\min_z p(z) := \frac{1}{2} \|Az - a\|^2 + \Phi(z).
\]
with $p$ bounded below, and $\Phi$ potentially nonsmooth, nonconvex, and non-finite valued. 
With step $\alpha = \frac{1}{\sigma_{\max}}$, the iterates~\eqref{eq:pg} satisfy 
\[
\min_{k = 0, \dots, N} \|v_{k+1}\|^2 \leq \frac{\|A\|^2}{N} (p(z_0) - \inf p)
\]
where 

\[
v_k = (\|A\|_2^2 I - A^TA)(x_k - x_{k+1})
\]
is a subgradient (generalized gradient) of $p$ at $z^k$.
\end{theorem}

We can specialize Theorem~\ref{thm:pg} to our case by computing the norm of the least squares system in~\eqref{eq:lsNC}.

\begin{corollary}[Rate for Algorithm~\ref{alg:prox-grad}]
\label{cor:proxgrad}
Theorem~\ref{thm:pg} applied to problem~\ref{eq:cost_1} gives 
\[
\min_{k = 0, \dots, N} \|v_{k+1}\|^2 \leq C(\eta_1, \eta_2, \cC, \cA) \frac{1}{N} (p(z_0) - \inf p)
\]
with 
\[
C(\eta_1, \eta_2, \cC, \cA) = \frac{1}{\eta_1} (c+\|\cC\|_F^2) + \frac{1}{\eta_2}(d + \|\cA\|_F^2).
\]
\end{corollary}
\begin{algorithm}[t!]
\caption{Value-function optimization for~\eqref{eq:cost_1}.}
\label{alg:varpro}
\begin{algorithmic}[1]
\State{\bfseries Input:} $x^0, w_1^0, w_2^0$
\State{Initialize: $k=0$}
\State{Define: $\cH = \frac{1}{\eta_1} \cC^T \cC + \frac{1}{\eta_2}\cA^T \cA$}
\While{not converged}
\Let{$x^{k+1}$}{$\cH^{-1}\left(\frac{1}{\eta_1}\cC^Tw_1^k + \frac{1}{\eta_2}\cA^T(b+w_2^k)\right)$}
%\Let{$w_{1}^{k+1}$}{$\prox_{\beta \eta_1 \phi}\left(\cC(x^{k+1})\right)$}
%\Let{$w_{2}^{k+1}$}{$\proj_{\sigma\mathbb{B}_{\psi}}\left(\cA(x^{k+1}) - b\right)$}
\Let{$w_{1}^{k+1}$}{$\prox_{\frac{\eta_1}{\beta} \phi}\left(w_1^k - \frac{\beta}{\eta_1} (w_1^k - \cC(x^{k+1}))\right)$}
\Let{$w_{2}^{k+1}$}{$\proj_{\sigma\mathbb{B}_{\psi}}\left(w_2^k - \frac{\beta}{\eta_2}(w_2 - (\cA(x^{k+1}) - b)\right)$}
\Let{$k$}{$k+1$}
\EndWhile
\State{\bfseries Output:} $w_1^k, w_2^k, x^k$
\end{algorithmic}
\end{algorithm}

Problem~\eqref{eq:cost_1} also admits a different optimization strategy, 
summarized in Algorithm~\ref{alg:varpro}.
We can formally minimize the objective in $x$ directly via the gradient, with the minimizer given by
\[
\begin{aligned}
x(w) &= \cH^{-1}\left( \frac{1}{\eta_1} \cC^T w_1 + \frac{1}{\eta_2} \cA^T (w_2+b)\right)\\
\cH & = \frac{1}{\eta_1} \cC^T \cC +  \frac{1}{\eta_2}\cA^T \cA
\end{aligned}
\]
with $w = (w_1, w_2)$. Plugging this expression back in gives a regularized least squares problem in 
$w$ alone:
\begin{equation}
\label{eq:value}
\begin{aligned}
\min_{w_1, w_2}  &p(w) :=  \phi(w_1) +\left\| \cF \begin{bmatrix}w_1\\ w_2 \end{bmatrix}-\tilde b\right\|^2  \quad \mbox{s.t.} \quad \psi(w_2)\leq \sigma \\
& \cF = \begin{bmatrix} \frac{1}{\sqrt{\eta_1}} \left(\frac{1}{\eta_1} \cC\cH^{-1}\cC^T - I \right)  & \frac{1}{\sqrt{\eta_1}\eta_2} \cC\cH^{-1}\cA^T \\
\frac{-1}{\sqrt{\eta_2}\eta_1} \cA\cH^{-1}\cC^T &  \frac{1}{\sqrt{\eta_2}} \left(I - \frac{1}{\eta_1} \cA\cH^{-1}\cA^T\right)\end{bmatrix}	\\
& \widetilde b = \begin{bmatrix} \frac{-1}{\sqrt{\eta_1}\eta_2} \cC\cH^{-1}\cA^Tb  \\
\frac{1}{\sqrt{\eta_2}} \left(\frac{1}{\eta_1}\cA\cH^{-1}\cA^T - I \right)b\end{bmatrix}.
\end{aligned}
\end{equation}
%where we can compute $F^TF$ to get the convergence of our algorithm.
Prox-gradient applied to the {\it value function} $p(w)$ in~\eqref{eq:value} with step $\beta$ gives the iteration 
\begin{equation}
\label{eq:pgv}
w^+ = \prox_{\frac1\beta \Phi} (w^k - \beta \cF^T(\cF w - \widetilde b))
\end{equation}
This iteration, as formally written, requires forming and applying the system $\cF$ in~\eqref{eq:value} at each iteration. 
In practice we compute the $x(w)$ update on the fly, as detailed in Algorithm~\ref{alg:varpro}. 
The equivalence of Algorithm~\ref{alg:varpro} to iteration~\eqref{eq:pgv} comes from the following derivative formula for value functions~\cite{bell2008algorithmic}: 
\[
\begin{aligned}
 \cF^T(\cF w - \widetilde b)) &= \frac{1}{\eta_1} \cC^T(\cC(x(w))-w_1) \\
 & + \frac{1}{\eta_2}\cA^T(\cA(x(w)) - (w_2 +b)). 
\end{aligned}
\]
In order to compute $\beta$, and apply Theorem~\ref{thm:pg}, we first prove the following lemma: 
\begin{lemma}[Bound on $\|\cF^T\cF\|_2$]

The operator norm $\|\cF^T\cF\|_2$ is bounded above by $\max \left(\frac{1}{\eta_1}, \frac{1}{\eta_2}\right)$.

\end{lemma}

\begin{proof}
Considering the function 
\[
\|\cF w- \widetilde b\|^2 = \min_x \underbrace{\frac{1}{2\eta_1} \|\cC(x)-w_1\|^2 + \frac{1}{2\eta_2}\|w_2 - \cA(x)+b\|_2^2}_{Q(x,w)},
\]
we know that the gradient is given by $\cF^T(\cF w - \widetilde b)$, and any Lipschitz bound $L$ 
gives 
\[
\|\cF^T \cF w_1 - \cF^T \cF w_2 \| \leq L \|w_1 - w_2\|,
\]
which means $\|\cF^T \cF\|_2 \leq L$. On the other hand, we can write the right hand side as 
\[
Q(w,x) = q(Dw,x)
\]
where 
\[
q(z,x) = \frac{1}{2}   \left\| z  -  \begin{bmatrix} \frac{1}{2\sqrt{\eta_1}} \cC(x) \\   \frac{1}{2\sqrt{\eta_2}} \cA(x)  \end{bmatrix} - \begin{bmatrix} 0 \\ b\end{bmatrix} \right\|^2
\]
and 
\[
D = \begin{bmatrix}  \frac{1}{\sqrt{\eta_1}} & 0 \\ 0 & \frac{1}{\sqrt{\eta_1}} \end{bmatrix}.
\]
Using Theorem 1 of \cite{fastnonsmooth} with $g(z) = 0$, we have  that the value function
\[
\widetilde q(z) = \min_x q(z,x)
\]
is differentiable, with $\mbox{lip}(\nabla \widetilde q) \leq 1$. Therefore 
\[
\widetilde Q(w) =  \min_x Q(w, x)
\]
is also differentiable, with 
\[
\nabla \widetilde Q(w) = D' \nabla \widetilde q(Dw),
\]
and hence 
\[
\mbox{lip} (\nabla \widetilde Q) \leq \|D^TD\|_2 = \max \left(\frac{1}{\eta_1}, \frac{1}{\eta_2}\right).
\]
This immediately gives the result. 
\end{proof}

Now we can combine iteration~\eqref{eq:pgv} with Theorem~\ref{thm:pg} to get a rate of convergence 
for Algorithm~\ref{alg:varpro}.

\begin{corollary}[Convergence of Algorithm~\ref{alg:varpro}]
\label{cor:conv}
When $\beta$ satisfies 
\[
\beta \leq \min(\eta_1, \eta_2),
\]
the iterates of Algorithm~\ref{alg:varpro} satisfy
\[
\min_{k = 0, \dots, N} \|v_{k+1}\|^2 \leq \frac{1}{N} \max \left(\frac{1}{\eta_1}, \frac{1}{\eta_2}\right)(p(w_0) - \inf p))
\]
where $v_k$ is in the subdifferential (generalized gradient) of objective~\eqref{eq:value} at $w^k$.
Moreover, if $\eta_1 = \eta_2$, then Algorithm~\eqref{alg:varpro} is equivalent to block-coordinate descent, 
as detailed in Algorithm~\ref{alg:bcd}. 

\end{corollary} 

\begin{proof}
The convergence statement comes directly from plugging the estimate of iteration~\ref{eq:pgv} into Theorem~\ref{thm:pg}.
The equivalence of Algorithm~\ref{alg:bcd} with Algorithm~\ref{alg:varpro} is obtained by plugging in step size 
$\beta = \eta_1 = \eta_2$ into each line of Algorithm~\ref{alg:varpro}.

\end{proof}

An important consequence of Corollary~\ref{cor:conv} is that the convergence rate of Algorithm~\ref{alg:varpro} 
does not depend on $\cC$ or $\cA$, in contrast to Algorithm~\ref{alg:prox-grad}, whose rate depends on 
both matrices (Corollary~\ref{cor:proxgrad}). The rates of both algorithms are affected by $(\eta_1, \eta_2)$. 
We use continuation in $\eta$, driving $(\eta_1, \eta_2)$ to $(0,0)$ at the same rate, and warm-starting each 
problem at the previous solution. A convergence theory that takes this continuation into account is left to future work.

\begin{algorithm}[t!]
\caption{Block-coordinate descent for~\eqref{eq:cost_1}.}
\label{alg:bcd}
\begin{algorithmic}[1]
\State{\bfseries Input:} $x^0, w_1^0, w_2^0$
\State{Initialize: $k=0$}
\State{Define: $\cH = \frac{1}{\eta_1} \cC^T \cC + \frac{1}{\eta_2}\cA^T \cA$}
\While{not converged}
\Let{$x^{k+1}$}{$\cH^{-1}\left(\frac{1}{\eta_1}\cC^Tw_1^k + \frac{1}{\eta_2}\cA^T(b+w_2^k)\right)$}
%\Let{$w_{1}^{k+1}$}{$\prox_{\beta \eta_1 \phi}\left(\cC(x^{k+1})\right)$}
%\Let{$w_{2}^{k+1}$}{$\proj_{\sigma\mathbb{B}_{\psi}}\left(\cA(x^{k+1}) - b\right)$}
\Let{$w_{1}^{k+1}$}{$\prox_{\phi}\left(\cC(x^{k+1})\right)$}
\Let{$w_{2}^{k+1}$}{$\proj_{\sigma\mathbb{B}_{\psi}}\left(\cA(x^{k+1}) - b)\right)$}
\Let{$k$}{$k+1$}
\EndWhile
\State{\bfseries Output:} $w_1^k, w_2^k, x^k$
\end{algorithmic}
\end{algorithm}

  \begin{table}[H]
  \begin{center}
      \caption{SNR values agains the true $x$ for different $\ell_p$ norms with Algorithm \ref{alg:bcd}.}
    \label{tab:small_ex}
    \begin{tabular}{|c|c|}
      \hline
      \multicolumn{2}{|c|}{BPDN with Random Linear Operator} \\
      \hline
      Method/Norm & SNR \\
      \hline
      $\ell_2$ with SPGL1 & 0.2007   \\
      $\ell_2$ with~Alg.\ref{alg:bcd} &   0.2032 \\\hline
      $\ell_1$ with~Alg.\ref{alg:bcd} & 33.7281   \\
      $\ell_\infty$ with~Alg.\ref{alg:bcd}& -0.6708    \\
      $\ell_0$ with~Alg.\ref{alg:bcd} &    45.0601 \\
      \hline
    \end{tabular}
    \end{center}
   \end{table}

\subsection{Inexact Least-Squares Solves.}

Algorithm~\ref{alg:bcd} has a provably faster rate of convergence than Algorithm~\ref{alg:prox-grad}. 
The practical performance of these algorithms is compared in Figure~\ref{fig:conv_rates}, 
which is solving a problem with both a $\ell_1$ norm regularizer and $\ell_1$ norm BPDN constraint, with $\alpha  = \|\cA\|_F^{-2}$, $\cC = I$, and $\eta_1 = \eta_2 = 10^{-4}$.
We see a huge performance difference in practice as well as in theory: the proximal gradient descent from Algorithm~\ref{alg:prox-grad} yields a slower cost function decay than solving exactly for $x(w)$ as in Algorithm~\ref{alg:bcd}. Indeed, Algorithm~\ref{alg:bcd} admits the fastest cost function decay as shown in Corollary~\ref{cor:conv}, albeit at the expense of more operations per iteration. This is due to the fact that fully solving the least squares problem in Line 5 is not tractable for large-scale problems. 
Hence, we implement Algorithm~\ref{alg:bcd} inexactly, 
using the Conjugate Gradient (CG) algorithm. 
Figure~\ref{fig:conv_rates} shows the results when we use 1, 5, and 20 CG iterations. 
Each CG iteration is implemented using matrix-vector products, and at 20 iterations 
the results are indistinguishable from those of Algorithm~\ref{alg:bcd} with full solves. 
Even at 5 iterations, the performance is remarkably close to that of 
of Algorithm~\ref{alg:bcd} with full solves. Algorithm~\ref{alg:bcd} 
has a natural warm-start strategy, with the $x$ from each previous iteration used in the subsequent LS solve using CG. 
Using a CG method with a bounded number of iterates gives fast convergence and saves computational time. 
This approach is used in the subsequent experiments.

% Figure \ref{fig:conv_rates} depicts a sample cost function decay for various hyper parameter choices in Algorithms \ref{alg:prox-grad}-\ref{alg:bcd} for a sample constructed BPDN problem.  

\section{Application to Basis Pursuit De-noise Models}\label{scn:bpdn}
The Basis Pursuit De-noise problem can be formulated as
\begin{equation}
\label{eq:bpnd}
\begin{aligned}
	\min_x \|x\|_1 \quad \text{s.t.}\
	\rho\left(\cA(x) - b\right) \leq \sigma
\end{aligned}
\end{equation}
where $\rho(\cdot)$ is classically taken to be the $\ell_2$-norm. In this problem, 
%$x\in\mathbb{C}^{n_s\times n_t}$ is the seismic dataset being penalized for sparsity, 
$x$ represents unknown coefficients that are sparse in a transform domain,  
%while $\mathcal{A}:\mathbb{C}^{n_s\times n_t}\rightarrow \mathbb{C}^p$ 
while $\mathcal{A}$ is a composition of the observation operator with a transform matrix; popular examples of transform domains include 
discrete cosine transforms, wavelets, and curvelets. 
%is a linear masking operator as well as a an operator that transforms $x$ out of the curvelet domain and into the temporal/spatial domain. 
The observed and noisy data $b$ resides in the temporal/spatial domain, and $\sigma$ is the misfit tolerance. 
This problem was famously solved with the SPGL1~\cite{spgl1} algorithm for $\rho(\cdot) = \|\cdot\|_2$. 
When the observed data is affected by large sparse noise, a smooth constraint is ineffective. 
A nonsmooth variant of~\eqref{eq:bpnd} is very difficult for approaches such as SPGL1, which solves subproblems of the form 
\[
\min_x \rho\left(\mathcal{A}(x) - b\right) \quad \mbox{s.t.} \quad \|x\|_1 \leq \tau.
\]
However, the proposed Algorithm ~\ref{alg:varpro} is easily adaptable to different norms. % due to the relaxation of constraint that is solved via projection.
% onto the $\ell_p$-ball. 
%and can find sparse $x$ and have the approach be robust to outliers.
We apply Algorithm~\ref{alg:bcd} with $\phi(x) = \|x\|_1$, 
taking $(\eta_1, \eta_2) \rightarrow (0,0)$ so that  $(w_1, w_2)\rightarrow (x, \mathcal{A}(x) - b)$. 
We can take many different $\psi$, including $\ell_2, \ell_1, \ell_\infty$, and $\ell_0$. 

%We take $\eta_1 = \eta_2$ at each stage of the continuation, so 
%Algorithm~\ref{alg:varpro} simplifies to Algorithm~\ref{alg:bcd}.

%\begin{algorithm}[H]
%\caption{Block-Coordinate Descent for~\eqref{eq:extobj}.}
%\label{alg:bcd}
%\begin{algorithmic}[1]
%\State{\bfseries Input:} $x^0, w_1^0, w_2^0$
%\State{Initialize: $k=0$}
%\While{not converged}
%\Let{$x^{k+1}$}{$\left(I+\frac{\eta_1}{\eta_2}\cA^*\cA\right)^{-1}(w_1^k + \frac{\eta_1}{\eta_2}\cA^*(b+w_2^k))$}
%\Let{$w_{1}^{k+1}$}{$\prox_{\eta_1 \phi}(x^{k+1})$}
%\Let{$w_{2}^{k+1}$}{$\proj_{\sigma\mathbb{B}_{\psi}}\lt(\cA(x^{k+1}) - b\rt)$}
%\Let{$k$}{$k+1$}
%\EndWhile
%\State{\bfseries Output:} $w_1^k, w_2^k, x^k$
%\end{algorithmic}
%\end{algorithm}

\begin{figure}[t!]
\centering     %%% not \center
\includegraphics[scale=0.13]{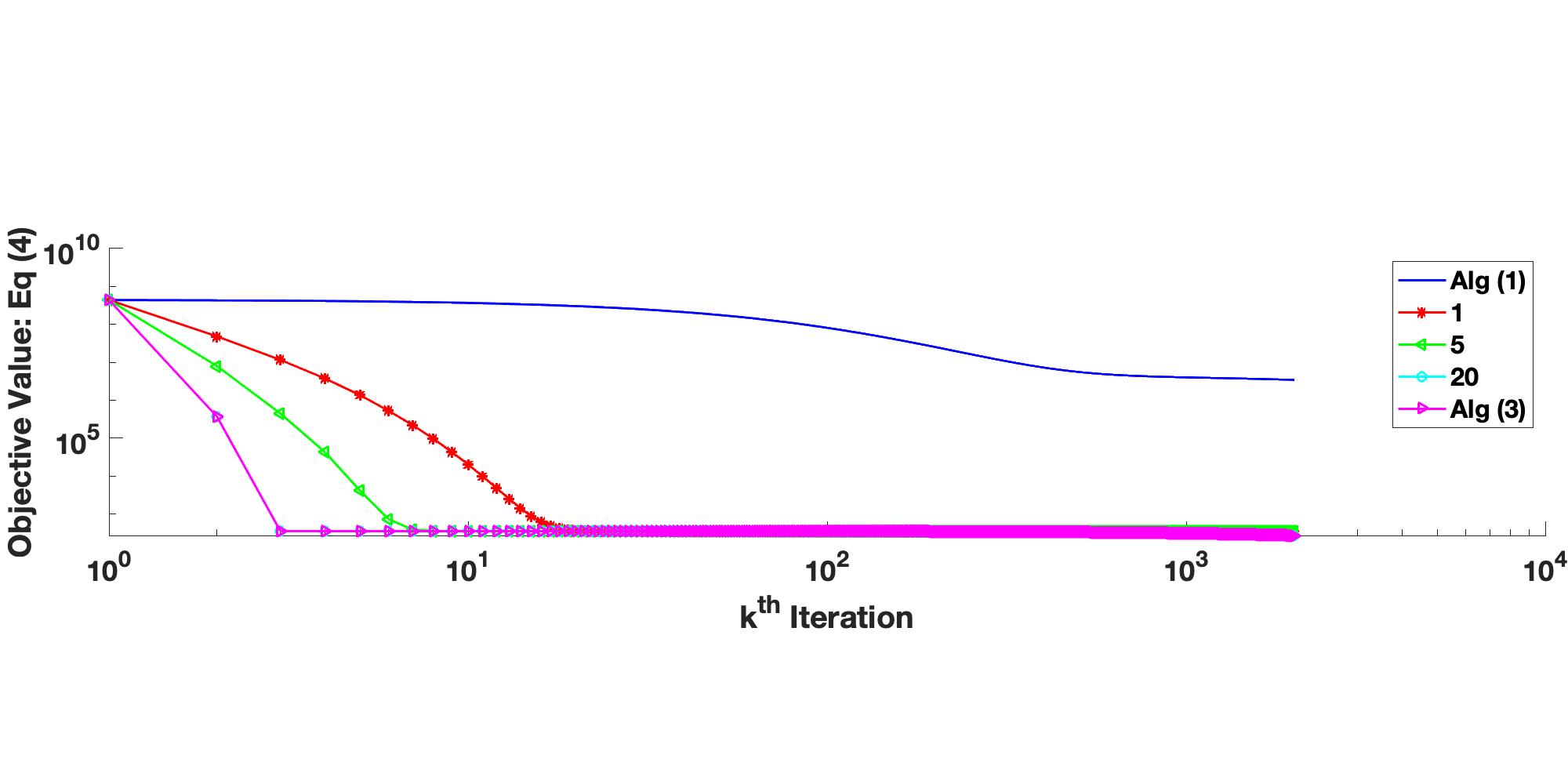}
\caption{Objective function decay for Equation \ref{eq:cost_1} with proximal-gradient descent (Algorithm \ref{alg:prox-grad}), Direct solving (Algorithm \ref{alg:bcd}), and several steps in between where we only partially solve for $\cH^{-1}(\hdots)$ with Algorithm \ref{alg:varpro}. }\label{fig:conv_rates}
\end{figure}

\begin{table*}[h]
\caption{\label{table:prox} Projectors for $\ell_p$ balls.}
\centering
\begin{tabular}{c | c | c | c }
Norm & $\ell(x)$ & $\proj_{\tau\mathbb{B}_\ell}(z)$ &  Solution \\ \hline\hline
$\ell_2$ & $ \sqrt{\sum_i x_i^2}$ & $\begin{cases} z, &\|z\| < \tau \\\tau z/\|z\|_2, & \|z\| > \tau \end{cases}$ & Analytic\\ \hline
$\ell_\infty$ & $\max_i |x_i|$ & $\max(\min(x,1),-1)$ & Analytic \\ \hline
$\ell_1$  & $\sum_i |x_i|$  & See e.g.~\cite{van2008probing}  & $O(n\ln n)$ routine \\ \hline
$\ell_0$ &  $\sum_i \mathbf{1}_{x_i \neq 0}$ & $\begin{cases} z_i, & i \mbox{ one of the } \tau \mbox{ largest indices} \\ 0 & \mbox{otherwise}.  \end{cases}$ & Analytic\\ \hline\hline
\end{tabular}
\end{table*}

\begin{figure}[t!]
\centering     %%% not \center
\subfigure[True]{\label{fig:true_res}\includegraphics[scale=0.13]{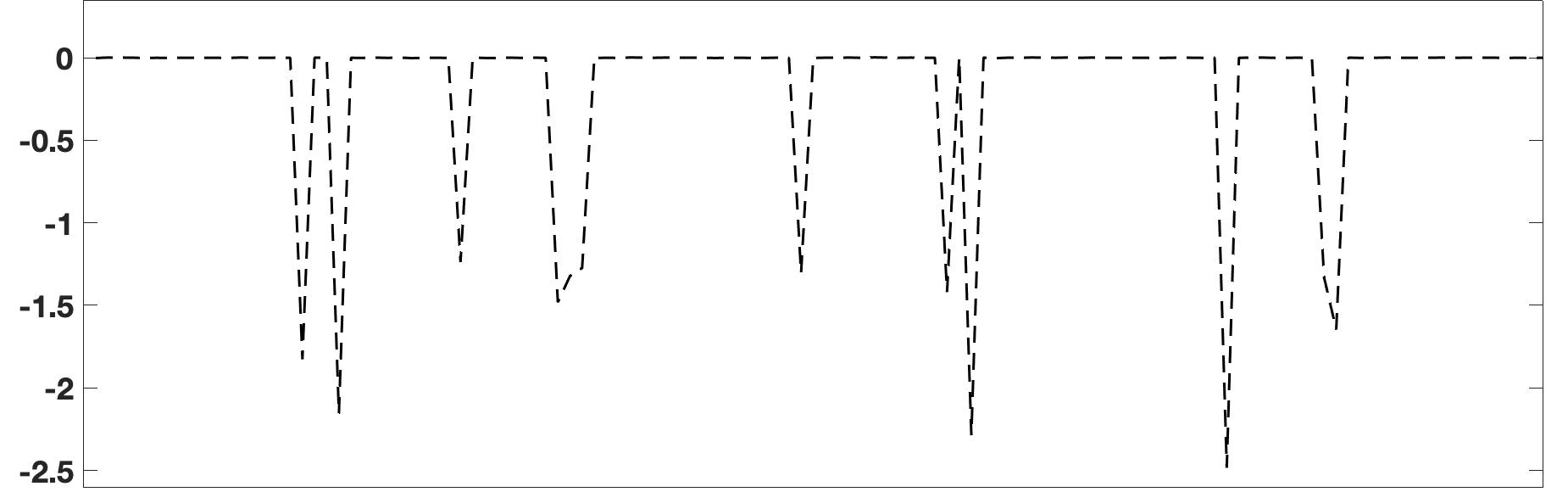}}
\subfigure[$\ell_2$]{\label{fig:l2_res}\includegraphics[scale=0.13]{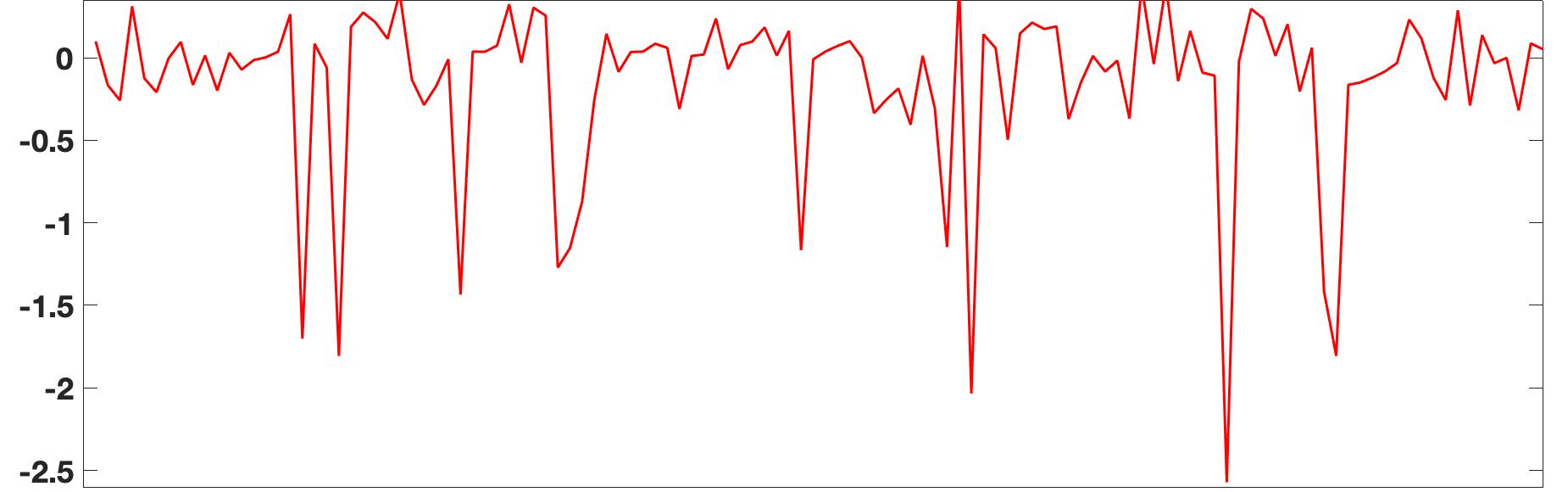}}
\subfigure[$\ell_1$]{\label{fig:l1_res}\includegraphics[scale=0.13]{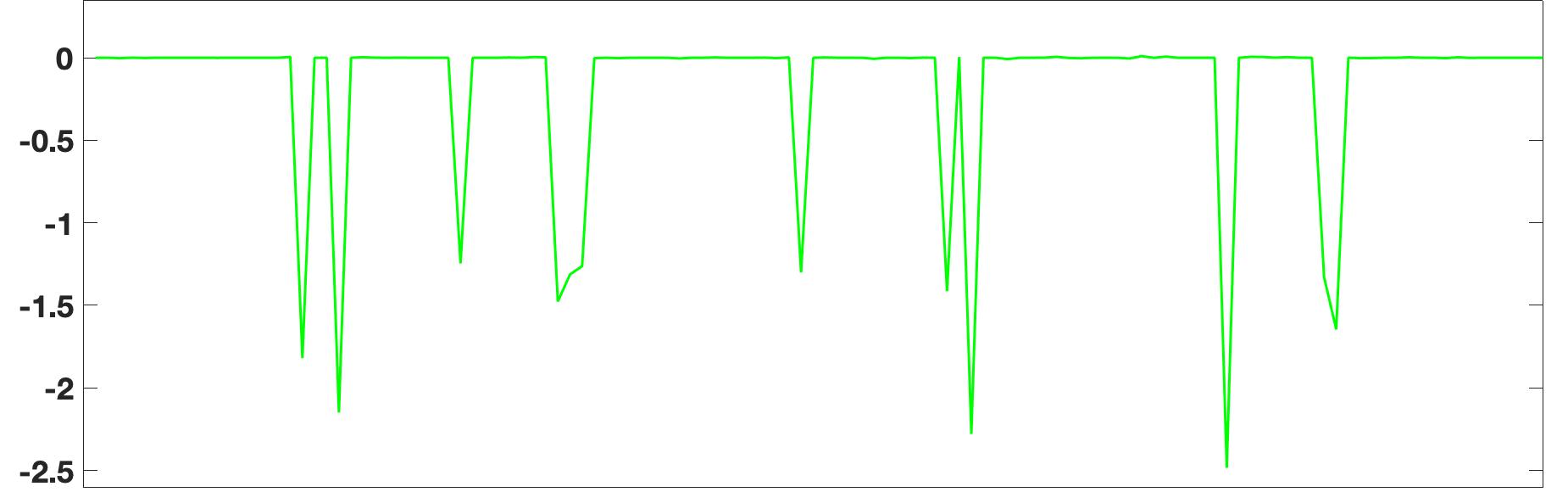}}
\subfigure[$\ell_\infty$]{\label{fig:linfty_res}\includegraphics[scale=0.13]{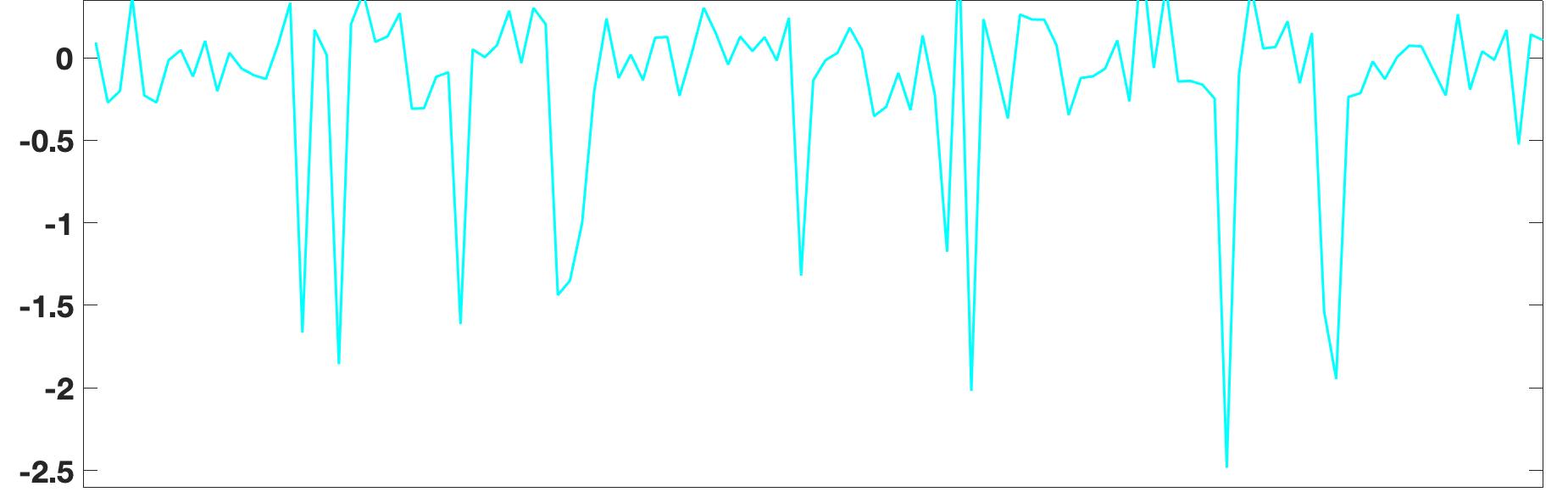}}
\subfigure[$\ell_0$]{\label{fig:l0_res}\includegraphics[scale=0.13]{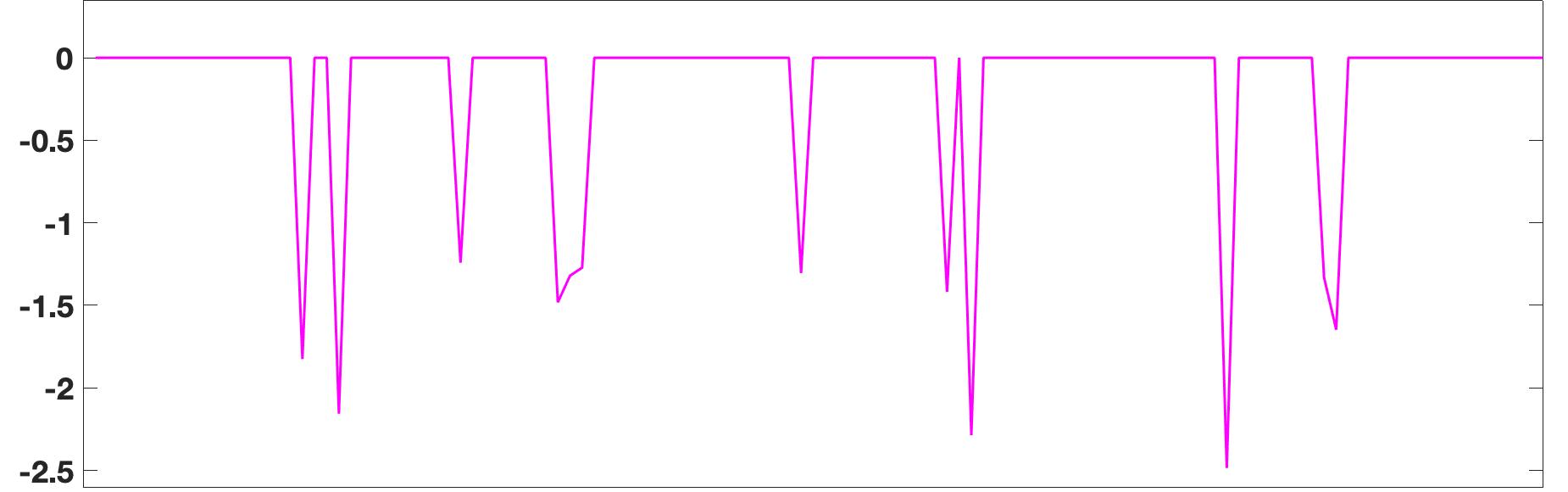}}
\caption{Residuals for different $\ell_p$-norms after algorithm termination. Note how the $\ell_{1}$- and $\ell_0$-norms can capture the outliers only.}\label{fig:se_res}
\end{figure}

\begin{figure}[t!]
\subfigure[True]{\label{fig:true_results}\includegraphics[scale=0.13]{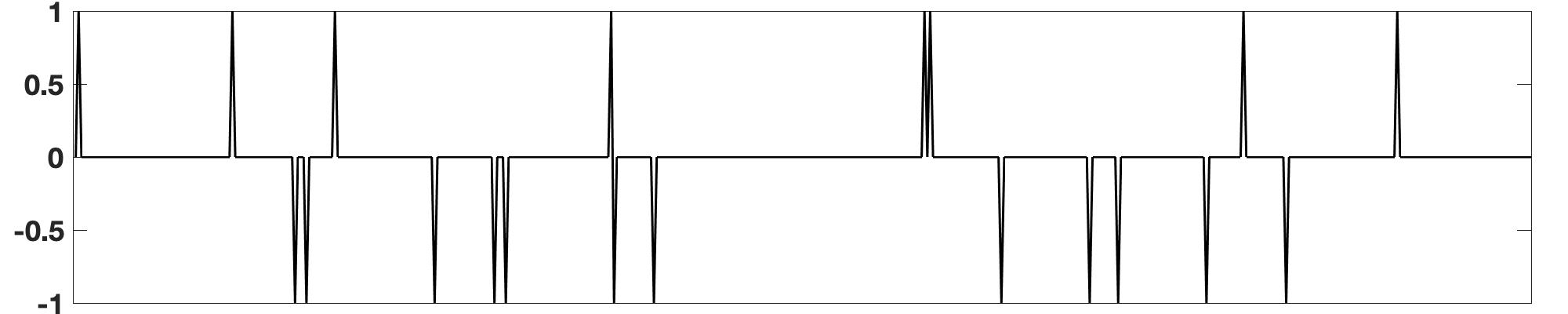}}
\subfigure[$\ell_2$]{\label{fig:l2_results}\includegraphics[scale=0.13]{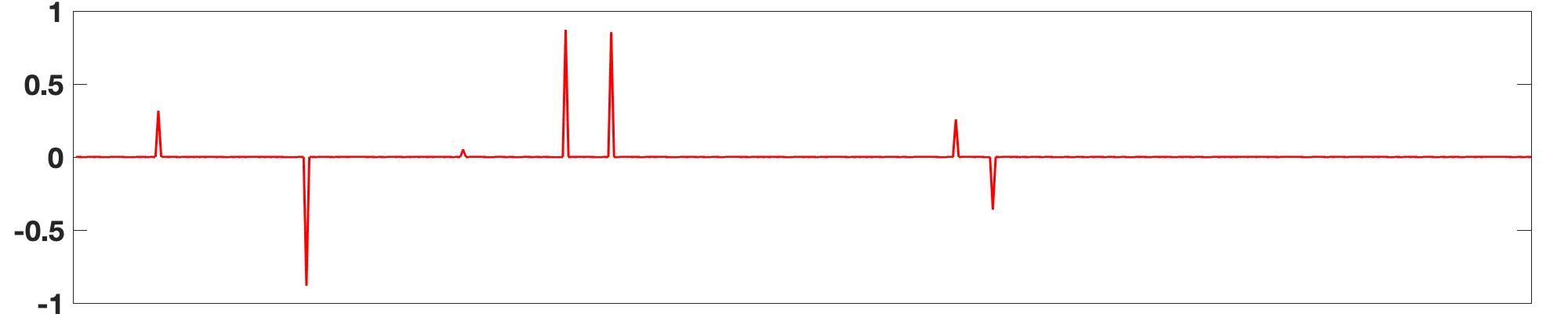}}
\subfigure[$\ell_1$]{\label{fig:l1_results}\includegraphics[scale=0.13]{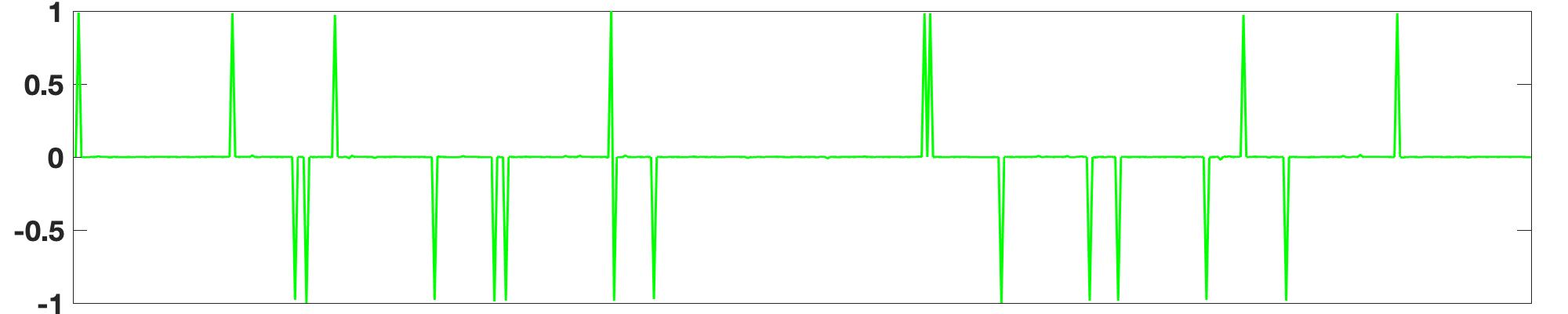}}
\subfigure[$\ell_\infty$]{\label{fig:linfty_results}\includegraphics[scale=0.13]{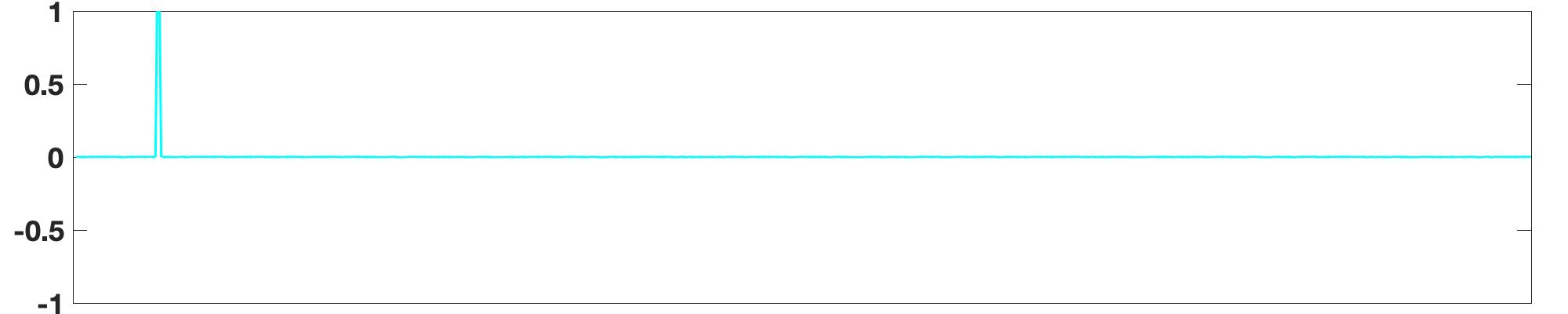}}
\subfigure[$\ell_0$]{\label{fig:l0_results}\includegraphics[scale=0.13]{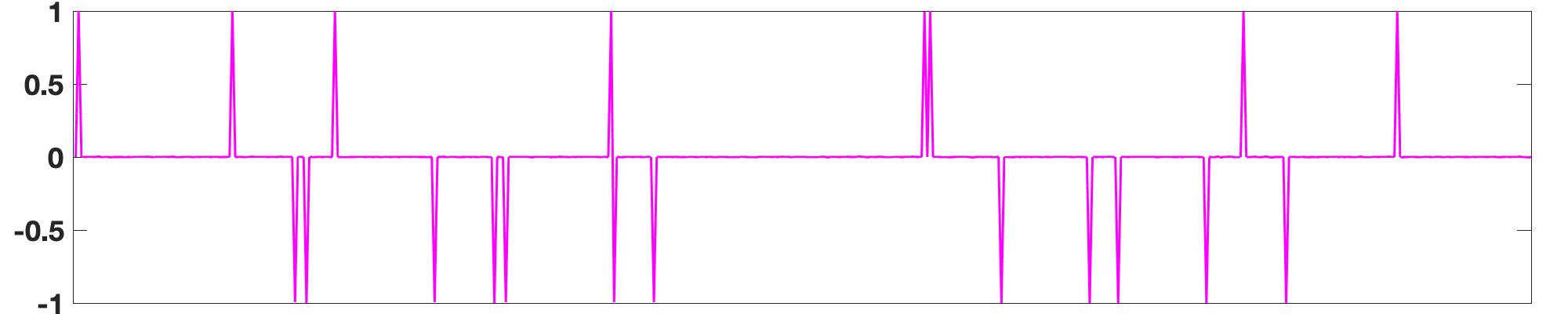}}
\caption{Basis Pursuit De-noising results for a randomly generated linear model with large, sparse noise.}
\label{fig:se_deno}
\end{figure}

Algorithm ~\ref{alg:bcd} is simple to implement. 
The least squares update in step 4  can be computed 
efficiently using either factorization with Woodbury, or an iterative method in cases where $\mathcal{A}$ is too large to store. 
%If $\cA(\cdot), \cC(\cdot)$ can be written as discrete matrix transformations, we refer to them as $C,\ A$ respectively. 
For the Woodbury approach, we have 
\begin{equation}
\label{eq:woodbury}
\left(\eta_2+\eta_1 \cA^T\cA\right)^{-1}  = \frac{1}{\eta_2}I - \frac{1}{\eta_2^2} \cA^T\left(\frac{1}{\eta_1} I + \frac{1}{\eta_2}\cA\cA^T\right)^{-1}\cA.
\end{equation}
For moderate size systems, we can store Cholesky factor
\[
LL^T =  \frac{1}{\eta_1} I + \frac{1}{\eta_2}\cA\cA^T,
\]
with $L \in \mathbb{R}^{m\times m}$, and use $L$ with~\eqref{eq:woodbury} 
to implement step 4. However, in the seismic/curvelet experiment described below, the left-hand side of Equation \ref{eq:woodbury} 
is too large to store in memory, but is positive definite. Hence, we solve the resulting linear system in step 4 of Algorithm ~\ref{alg:bcd} with CG, using matrix-vector products. 
The $w_1$ update is implemented via the $\ell_1$-proximal operator (soft thresholding), while the $w_2$ update requires a projection onto the $\ell_p$ ball. 
The projectors used in our experiments are collected in Table~\ref{table:prox}. 

The least  squares solve for $x$ is when $\cC^T$ is an orthogonal matrix or tight frame, so that $\cC^T\cC = I$; this is the case 
for Fourier transforms, wavelets, and curvelets. When $\cA$ is a restriction operator, as for many data interpolation problems, $\cA^T\cA$ is a diagonal 
matrix with zeros and ones, and hence
\[
\cH = \frac{1}{\eta_1} \cC^T \cC + \frac{1}{\eta_2} \cA^T\cA
\]
is a diagonal matrix with entries either $\frac{1}{\eta_1}$  or $\frac{1}{\eta_1} + \frac{1}{\eta_2}$; the least squares problem for the $x$ update is then trivial.

%Note that however, we can also easily flip between keeping $x$ in the transform domain or the spatial/temporal domain. Assuming formulation \ref{eq:flip_phys} this simplifies to the cost function described by Equation \ref{eq:cost_2}, thereby yielding updates described in \ref{eq:derivs_full}. Thus, Algorithm \ref{alg:bcd} can be modified thusly
%where now, $\cA(x)$ is a restriction operator that takes $x$ to the observed positions at $b$; hence we have that $A^TA = I$ is the identity matrix. If $\cC$ is certain transform operators, such as the curvelet transform, then $C^TC = I$ the also identity matrix. Hence, the update for $x^{k+1}$ becomes the much simpler diagonal matrix inversion. 

\section{Basis Pursuit De-Noise Experiments}\label{scn:bpdn_test}
In this application, we consider two examples: the first is a small-scale BPDN to illustrate the proof of concept of our technique, while the second is an application to de-noising a common source gather extracted from a seismic line simulated using a 2D BG Compass model. The data set contains time samples with a temporal-interval of 4ms, and the spatial sampling is 10m. For this example, we use curvelets as a sparsfying transform domain. 
The first example considers the same model as in~\eqref{eq:bpnd} where we want to enforce sparsity on $x$ while constraining the data misfit. 
The variable $x$ is a vector of length $n$ that has values $\{-1, 1\}$ on a random $4\%$ of its entries and zeros everywhere else; 
represents a spike train that we observe using a linear operator, $A\in\R^{n,m}$. 
$A$ was generated with independent standard Gaussian entries, and $b\in\R^m$ is observed data with large, sparse noise. 
We take $m = 120$ and $n = 512$. The noise is generated by placing large values on 10\% of the observations and assuming everything else was observed cleanly (ie no noise). Here, we test the efficacy of using different $\ell_p$ norms on the residual constraint. With the addition of large, sparse noise to the data, smooth norms on the residual constraint should not be able to effectively deal with such outlier residuals. With our adaptable formulation, it should be easy to enforce both sparsity in the $x$ domain as well as the residuals. Other formulations, such as SPGL1, do not have this capability. 

This noise is depicted in as the bottom black dashed line in Figure \ref{fig:se_res}. The results are shown in Figure~\ref{fig:se_deno} and in Table~\ref{tab:small_ex}. From these, we can clearly see that the $\ell_2$ norm is not effective for sparse noise, even at the correct error budget $\sigma$. 
Our approach is resilient to different types of noise since we can easily change the residual ball projection. 
This is seen by the almost exact accuracy of the $\ell_1$ and $\ell_0$ norms, with SNR's of 33 and 45 respectively. \\

The next test of the BPDN formulation is for a common source gather where entries are both omitted and corrupted with synthetic noise. Here, the objective function looks for sparsity in the curvelet domain, while the residual constraint seeks to match observed data within a certain tolerance $\sigma$. First, we note that doing interpolation only without added noise yields an SNR of approximately 13 for all formulations and algorithms; that is, all $\ell_p$ norms for Algorithm \ref{alg:block-descent} and SPGL1. Here, we again want to enforce sparsity both in the curvelet domain ($x$) and the data residual ($\|\cA(x) - b\|$), which SPGL1 and other algorithms lack the capacity to do.

Following the first experiment, we add large sparse noise to a handful of data points; 
in this case, we added large values to a random 1\% of observations (this does not include omitted entries). 
The noise added is approximately 120, while the observed data can range from 0 to 30. 
The interpolated and denoising results are shown in Figure \ref{fig:curv} and Table \ref{tab:curv}. 
Large, sparse noise cannot be filtered effectively by a smooth norm constraint, using either Algorithm \ref{alg:block-descent} or SPGL1. 
However, $\ell_1$ and $\ell_0$ norms effectively handle such noise, and can be optimized using our approach.
The SNR's for these implementations are approximately 10 and 11 respectively, approaching that of the noiseless data mentioned above. 

\begin{table}
  \centering
  \caption{Curvelet Interpolation and Denoising results for SPGL1 and Algorithm \ref{alg:block-descent} for selected $\ell_p$-norms for BPDN.}
  \label{tab:curv}
    \begin{tabular}{|c|c|c|c|}
      \hline
      \multicolumn{4}{|c|}{4D Monochromatic Interpolation} \\
      \hline
      Method/Norm & SNR & SNR $w_1$ & Time (s) \\
      \hline
      $\ell_2$ with SPGL1 & 1.4594 &  - & 52.5 (early stoppage)\\
      $l_2$ with~Alg.\ref{alg:block-descent} &   0.1420 & 0.0851& 1348 \\\hline
      $l_1$ with~Alg.\ref{alg:block-descent} & 13.0193 & 12.5768 & 1335\\
      $l_\infty$ with~Alg.\ref{alg:block-descent}& 0.0000  & 0 & 776\\
      $l_0$ with~Alg.\ref{alg:block-descent} &    13.9019 & 13.4294& 1120 \\
      \hline
    \end{tabular}
\end{table}            

\begin{figure}
\centering
\subfigure[True Data]{\label{fig:true_curv}\includegraphics[scale=0.1]{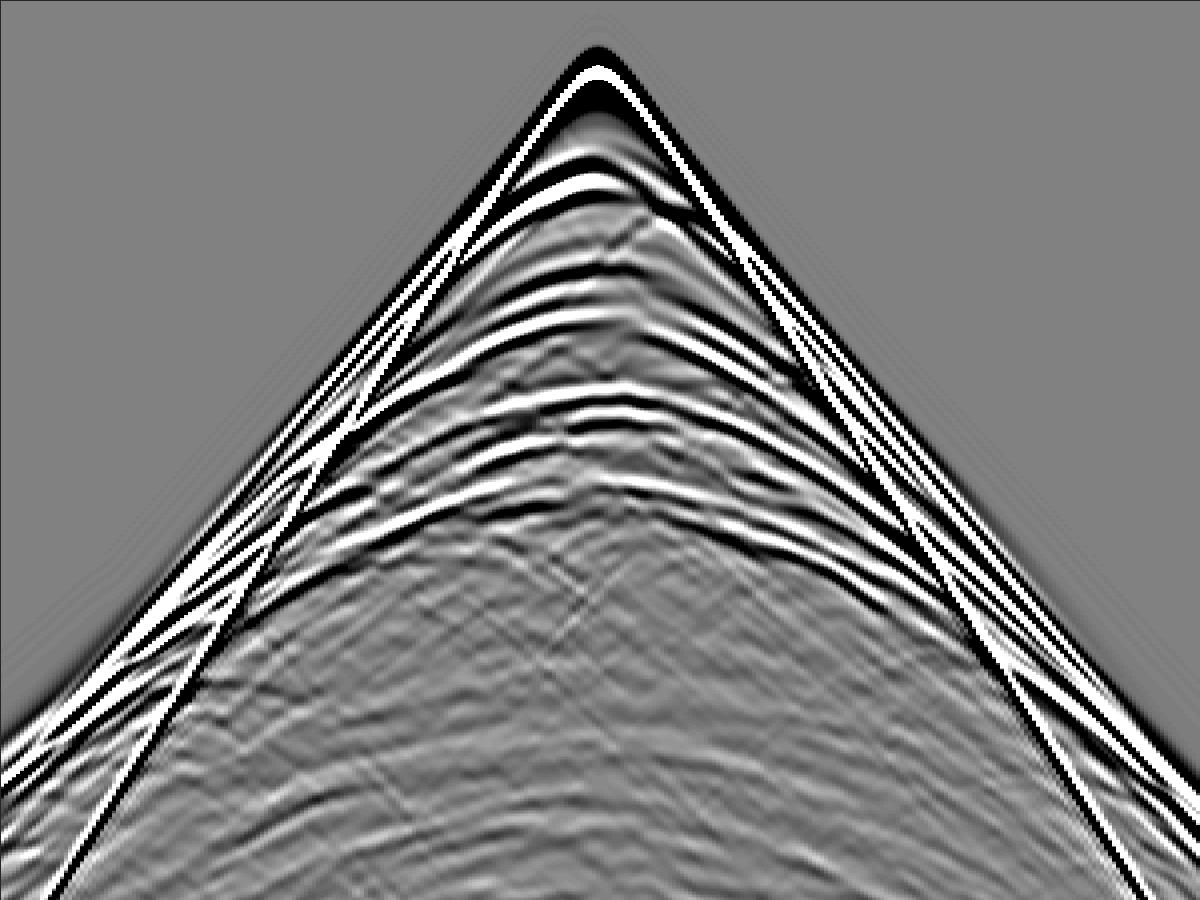} }
\subfigure[Added  Noise (binary)]{\label{fig:noise_curve}\includegraphics[scale=0.1]{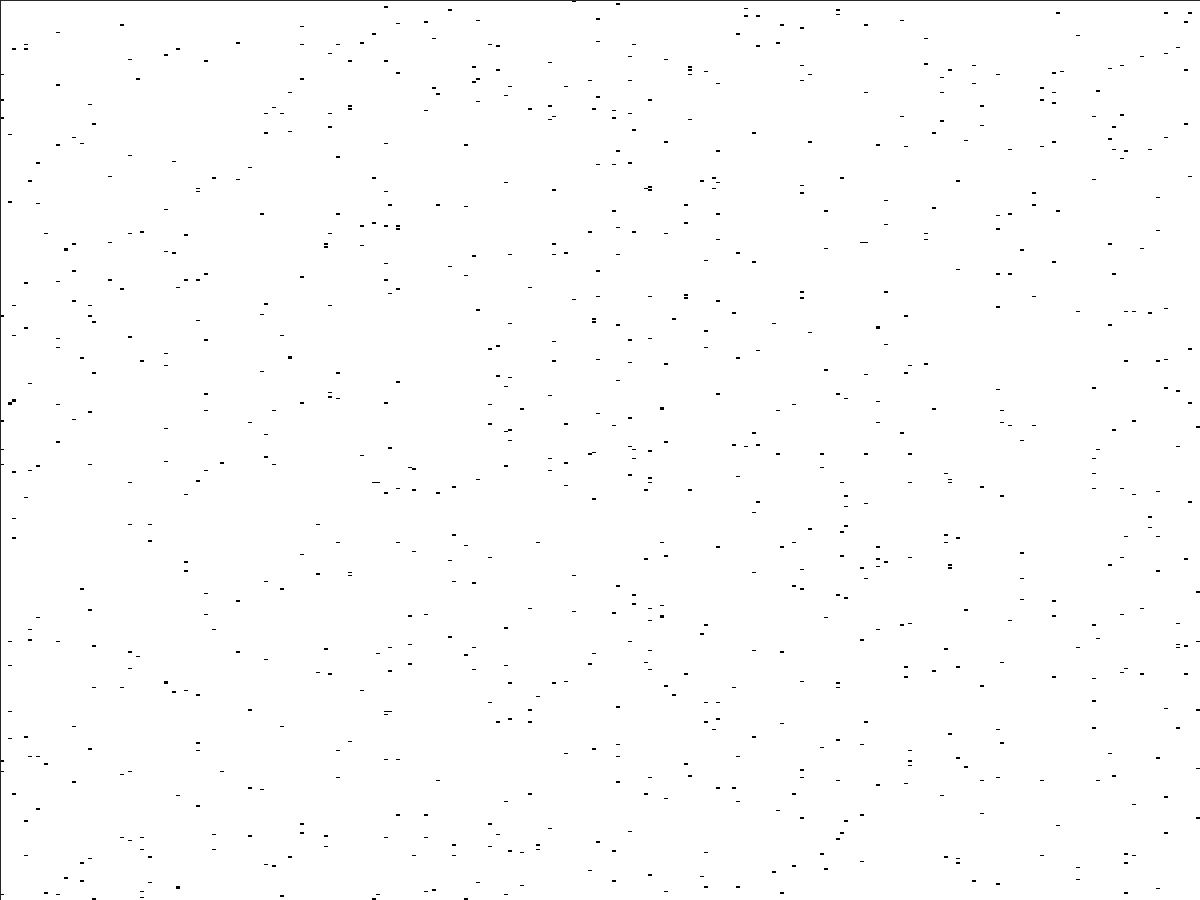}}
\subfigure[Noisy Data with Missing Sources]{\label{fig:noise_curv}\includegraphics[scale=0.1]{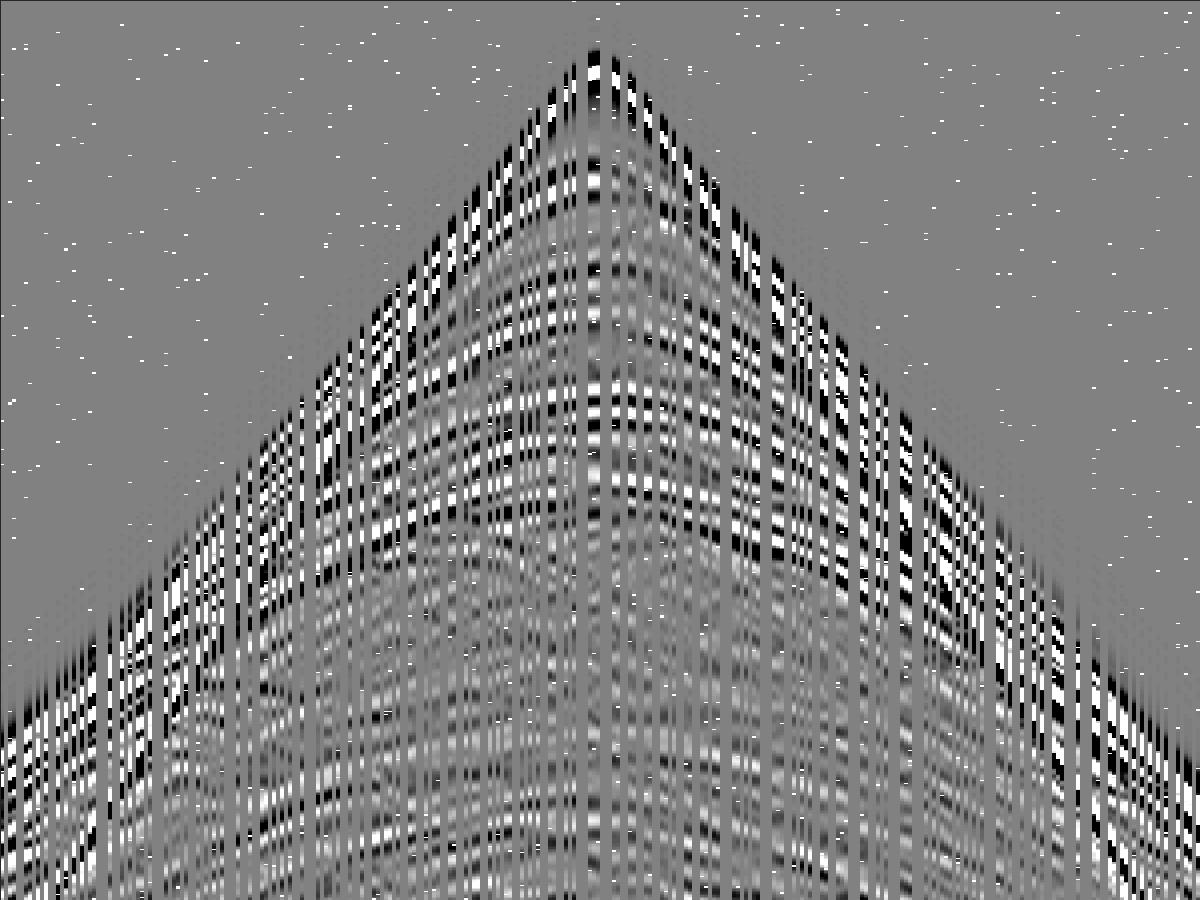} }
\subfigure[SPGL1]{\label{fig:spgl1_curv}\includegraphics[scale=0.1]{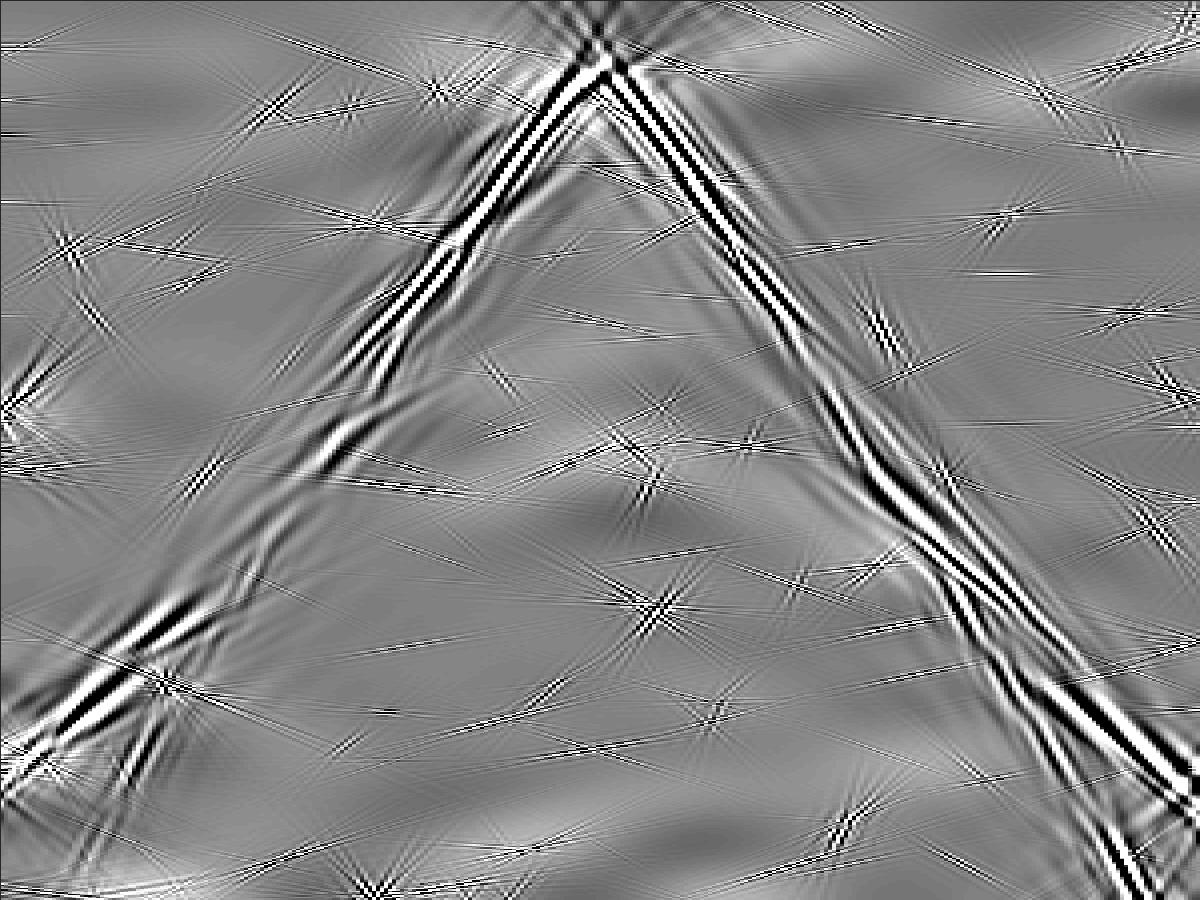} }
\subfigure[$l_2$]{\label{fig:l2_curve}\includegraphics[scale=0.1]{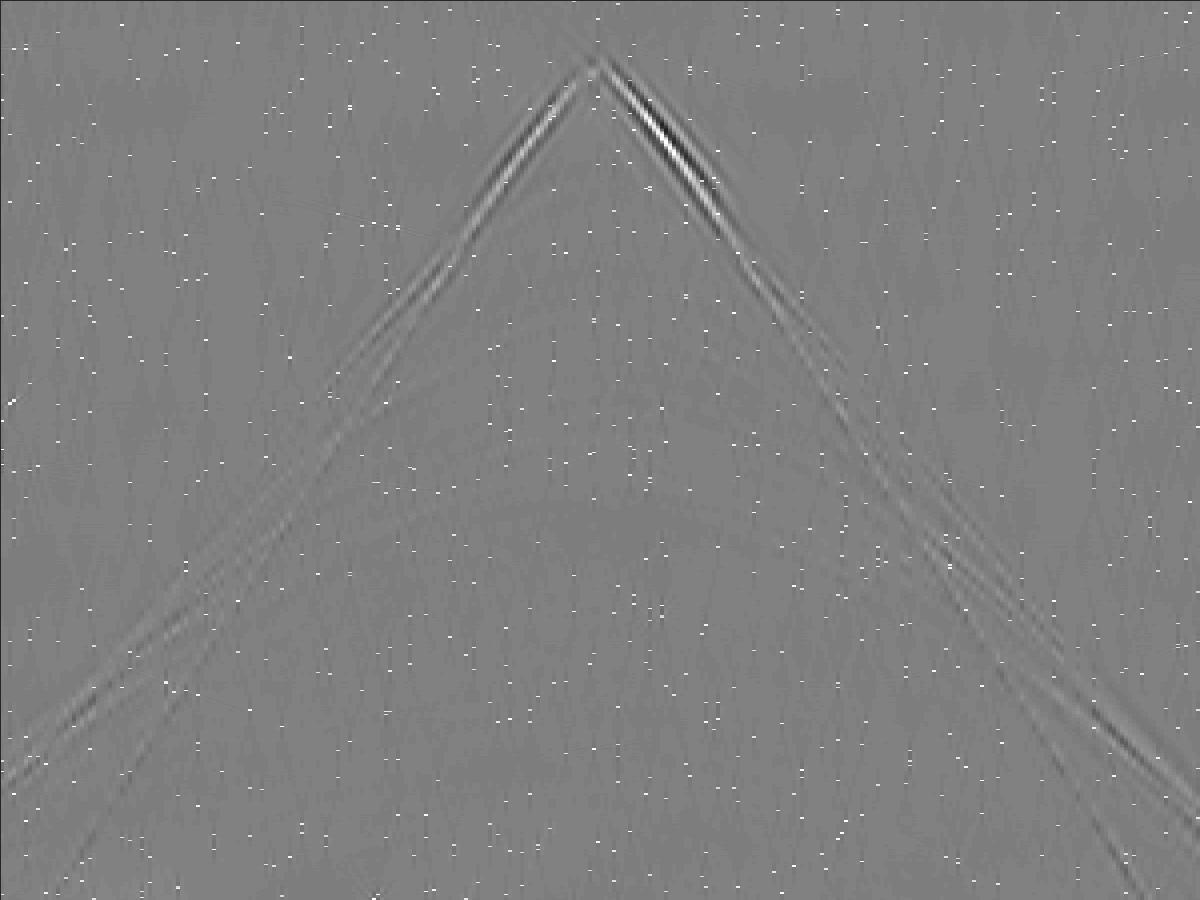}}
\subfigure[$l_1$]{\label{fig:l1_curve}\includegraphics[scale=0.1]{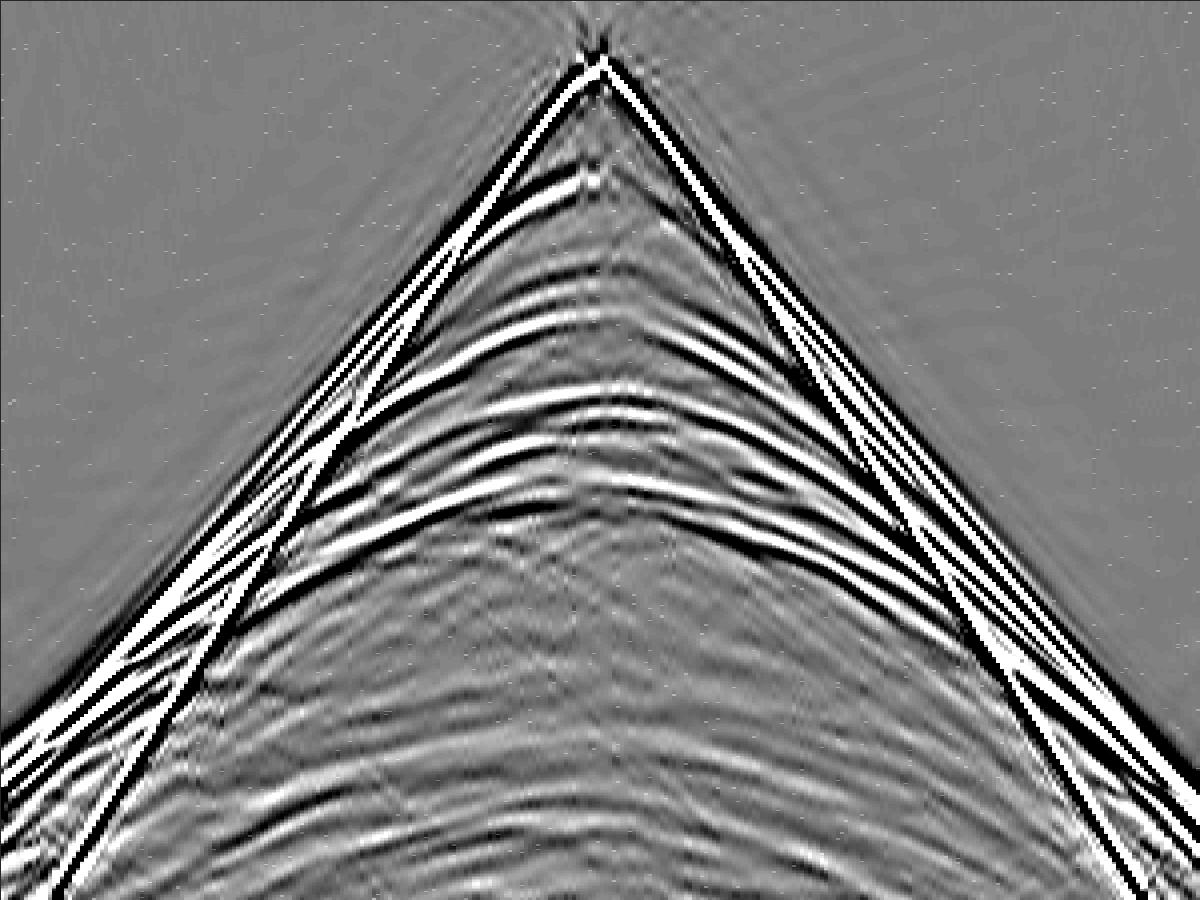}}
\subfigure[$l_\infty$]{\label{fig:linf_curve}\includegraphics[scale=0.1]{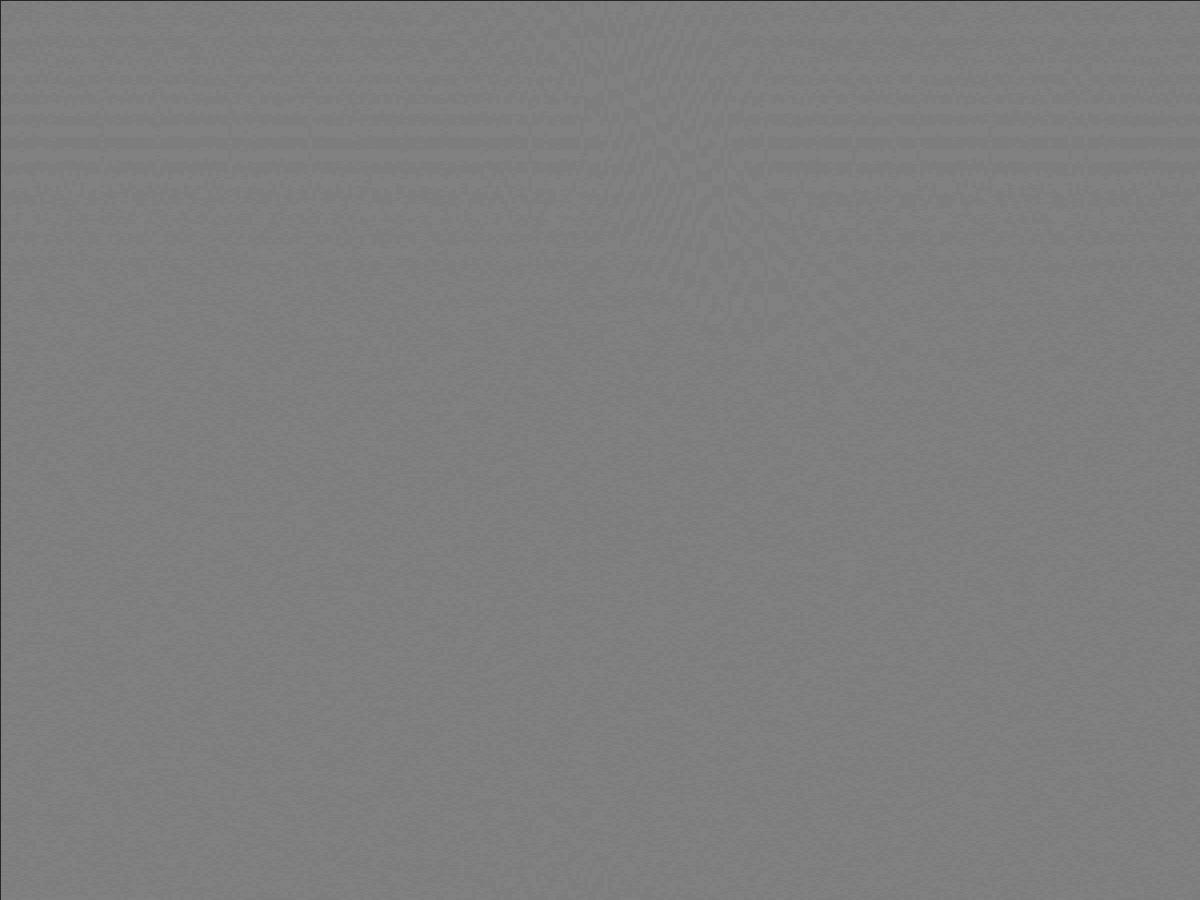}}
\subfigure[$l_0$]{\label{fig:l0_curve}\includegraphics[scale=0.1]{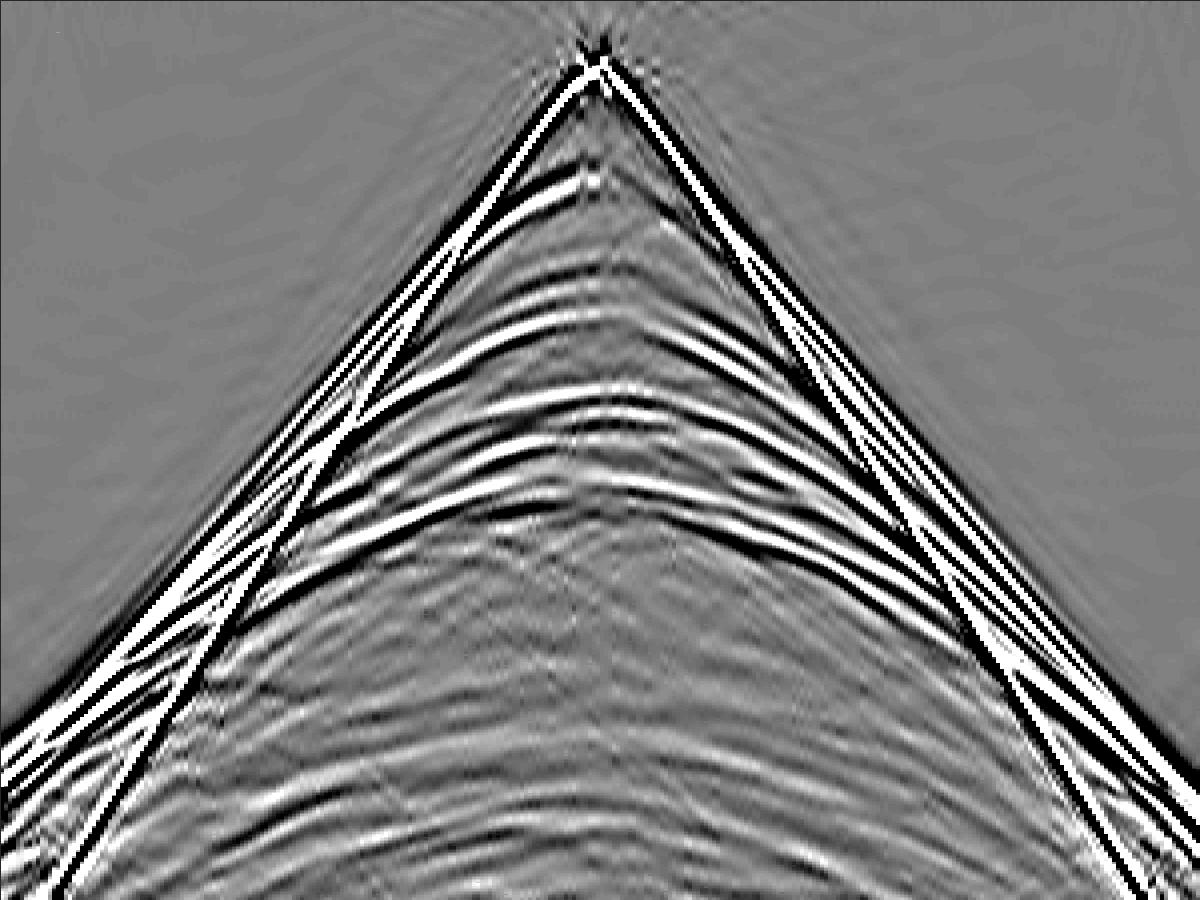}}
\caption{Interpolation and de-noising results for BPDN in the curvelet domain. Observe the complete inaccuracy of smooth norms with large, sparse noise.}
\label{fig:curv}
\end{figure}

\section{Extension to Low-Rank Models}\label{scn:lr}

Treating the data as having a matrix structure gives additional regularization tools --- 
in particular low-rank structure in particular domains. 
The BPDN formulation for residual-constrained low-rank interpolation 
is given by
\begin{equation}
  	\label{eqn:morozov}
  	\min_{X}\|X\|_* \quad \text{s.t.}\ \rho\left(\mathcal{A}(X) - b\right) \leq \sigma
  \end{equation} 
 for $X\in \mathbb{C}^{m\times n}$, $\mathcal{A}:\mathbb{C}^{n\times m} \rightarrow \mathbb{C}^p$ is a linear masking operator from full to observed (noisy) data $b$, 
and $\sigma$ is the misfit tolerance. 
The nuclear norm $\|X\|_*$ is the $\ell_1$ norm of the singular values of $X$. 
Solving the problem~\eqref{eqn:morozov} requires using a decision variable that 
is the size of the data, as well as updates to this variable that require SVDs at each iteration. 
It is much more efficient to model $X$ is a product of two matrices $L$ and $R$, given by
\begin{equation}
	\min_{L, R} \frac{1}{2}(\|L\|_F^2 + \|R\|_F^2) \quad \text{s.t.} \ 
\rho\left(\mathcal{A}(LR^T) - b\right) \leq \sigma \label{eqn:lrform}
\end{equation}
where $L\in\mathbb{C}^{n\times k}$, $R\in\mathbb{C}^{m\times k}$, and $LR^T$ is the low-rank representation of the data. 
The solution is guaranteed to be at most rank $k$, and in addition, 
the regularizer $\frac{1}{2}(\|L\|_F^2 + \|R\|_F^2) $ is an upper bound for $\|LR^T\|_*$, the sum of singular values of 
$LR^T$, further penalizing rank by proxy.  
The decision variables then have combined dimension $k(m\times n)$, which is much smaller than the $nm$ variables 
required by convex formulations. When $\rho$ is smooth, the problems are solved using a continuation  
that interchanges the roles of the objective and constraints, solving a sequence of 
problems where $\rho\left(\mathcal{A}(LR^T) - b\right)$ is minimized over the $\ell_2$ ball~\cite{fastlowrank} using 
projected gradient; an approach we call SPGLR below.  

When $\rho$ is not smooth, SPGLR does not work and there are no available implementations for~\eqref{eqn:lrform}. 
Nonsmooth $\rho$ arise when we want the residual to be in the $\ell_1$ norm ball, so we are 
 robust to outliers in the data, and can exactly fit inliers. 

We now extend Algorithm~\ref{alg:bcd} to this case. 
For any $\rho$ (smooth or nonsmooth), we introduce a latent variable $W$ for the data matrix, and solve
%
%To maneuver around this, we rephrase the problem to match that in
%\begin{equation}
%	\min_x h(F(x)) + g(x)
%\end{equation}
%where $h(x), F(x)$ and $g(x)$ are non-smooth, possibly non-convex functions. The idea is to substitute the “difficult” part of the problem out with a projected variable $w$, and then optimize the difference between the projected variable and the original function: 
%\begin{equation}
%	\min_{x,w} h(w) + \frac{1}{2\eta}\|w - F(x)\|_2^2 + g(x).
%\end{equation}
%Convergence results are given for $g(x)$ convex, $F(x)$ linear, and $h(x)$ prox-bounded in \cite[]{fastnonsmooth}. Note that these classes to not include low-rank interpolation, but the results seem to be promising and is an avenue for future theoretical work. Using the same technique on low-rank matrix completion gives us the problem formulation
\begin{equation}
 \label{eqn:final}
	\min_{L, R, W} \left\|\begin{matrix} L \\ R\end{matrix}\right\|_F^2 + \frac{1}{2\eta}\|W - LR^T\|_2^2, \quad \text{s.t.} \ \|\mathcal{A}(W) - b\|_p \leq \sigma
\end{equation}
with $\eta$ a parameter that controls the degree of relaxation; as $\eta \downarrow 0$ we have $W \rightarrow LR^T$. 
The relaxation allows a simple block-coordinate descent detailed in Algorithm~\ref{alg:block-descent}.
\begin{algorithm}[H]
\caption{Block-Coordinate Descent for~\eqref{eqn:final}.}
\label{alg:block-descent}
\begin{algorithmic}[1]
\State{\bfseries Input:} $w_0, L_0, R_0$
\State{Initialize: $k=0$}
\While{not converged}
\Let{$L_{k+1}$}{$\left(I +\eta R_k^TR_k\right)^{-1}(\eta W_kR_k)$}
\Let{$R_{k+1}$}{$(\eta W_k^T L_{k+1})\left(I + \eta L_{k+1}^TL_{k+1}\right)^{-1}$}
\Let{$W_{k+1}$}{$\begin{cases}(L_{k+1}R^T_{k+1})_{ij},\quad  (i,j) \in X_{obs}\\\proj_{\mathbb{B}_{\rho,\sigma}}\left(\mathcal{A}(L_{k+1}R_{k+1}^T)-b\right), \quad  \text{o.w.}\end{cases}$}
\Let{$k$}{$k+1$}
\EndWhile
\State{\bfseries Output:} $w_k, L_k, R_k$
\end{algorithmic}
\end{algorithm}
Algorithm~\ref{alg:block-descent} is also simple to implement. It requires two least squares solves (for $L$ and $R$), 
which are inherently parallelizable. It also requires a projection of the updated data matrix estimates $LR^T$
onto the $\sigma$-level set of the misfit penalty $\rho$. This step is detailed below. 

For unobserved data $(i,j)\not\in X_{obs}$, we have $W_{ij} = (LR^T)_{ij}$. 
For observed data, let $v$ denote $\mathcal{A}(LR^T)$.
Then the $W$ update step is given by solving 
\begin{align*}
	\min_w \|w - v\|^2_2, \quad \text{s.t.} \|w - b\|_p\leq \sigma.
\end{align*}
Using the simple substitution $z  = w-b$, the we get 
\begin{align*}
	\min_z \|z - (v -b)\|_2^2, \quad \text{s.t.} \quad \|z\|_p \leq \sigma
\end{align*}
which is precisely the projection of $\mathcal{A}(LR^T) - b$ onto $\mathbb{B}_{p,\sigma}$, the $\sigma$-level set of $\rho$. 
We use the same projectors for $\rho \in \{l_0, l_1, l_2, l_\infty\}$ as in Section~\ref{scn:bpdn_test}, see Table~\ref{table:prox}. 
%As an illustration, $p=2$ implies we project $z$ onto the $2$-norm ball 
%\begin{align*}
% 	z = \sigma \frac{(v-b)}{\|v-b\|_2}
% \end{align*} 
% meaning that $w = \sigma \frac{(v-b)}{\|v-b\|_2} +b$ for $(i,j)\in X_{obs}$. Hence, the minimum for the 2-norm is given by 
% \begin{align*}
% 	w_{ij} = \begin{cases}
%		(LR^T)_{ij}, \quad \text{for} \ (i,j)\not\in X_{obs}\\
%		\left(\sigma \frac{(\mathcal{A}(LR^T)-b)}{\|\mathcal{A}(LR^T)-b\|} + b\right)_{ij} \quad \text{for} \ (i,j)\in X_{obs}.
%	\end{cases}
% \end{align*}
The convergence criteria for Algorithm~\ref{alg:block-descent} is based on the optimality of the quadratic subproblems in $L,R$ and feasibility measure of $W-LR^T$, 
though in practice we compare performance of algorithms based on a computational budget.  
This block-coordinate descent scheme converges to a stationary point of Equation~\ref{eqn:final} by \cite[Theorem 4.1]{tseng2001convergence}.

Implementing block-coordinate descent on these forms until convergence produces the completed low-rank matrix. 
Setting $\nu = \|LR^T - w\|_2^2$, we iterate until $\nu< 1e-5$ or a maximum number of iterations is reached. 
%Neither of these hyperparameters need to be tuned.
In the next section, we develop an application of this method to seismic interpolation and denoising.

 \section{4D Matrix completion with De-noising}\label{scn:lr_test}
There are two main requirements when using the rank-minimization based framework for seismic data interpolation and denoising:  
 \emph{(i)} underlying seismic data should exhibit low-rank structure (singular values should decay fast) in some transform domain, and, \emph{(ii)} subsampling and noise destroy the low-rank structure (singular values decay slow) in that domain. For exploiting the low-rank structure during interpolation and denoising, we follow the matricization strategy proposed by \cite{da2015optimization}. The matricization (source-x, source-y), i.e., placing both the source coordinates along the columns (Figure~\ref{fig:sxsy}), 
 gives slow-decay of singular values (Figure~\ref{fig:svalsa}), while the matricization (source-x, receiver-x) (Figure~\ref{fig:sxrx}) gives fast decay of the singular values (Figure~\ref{fig:svalsb}).  To understand the effect of subsampling on the low-rank structure, we remove the  50\% of the sources. 
 Subsampling destroys the fast singular value decay in the (source-x, receiver-x) matricization, but not in the (source-x, receiver-y) matricization. 
This is because missing sources are missing columns in the (source-x, source-y) matricization, and missing sub-blocks in the (source-x, receiver-x) matricization (Figure~\ref{fig:sxsy_sub}). The latter is more effective for low-rank interpolation.

Similar to the BPDN experiments, we want to show that nonsmooth constraints on the data residual can be effective for dealing with large, sparse noise. The smooth $\ell_2$ norm that is most common in BPDN problem will fail in such examples, thereby leading to better data estimation with the implementation of non-smooth norms on the residuals. Thus, the goal of the below experiments is to show that enforcing sparsity in the singular values (ie low-rank) and sparsity in the residual constraint can be more effective with large, sparse noise than smooth residual constraints solved by most contemporary algorithms.

 \subsection{Experiment Description}
This example demonstrates the efficacy of the proposed approach using data created by a 5D dataset based on a complex SEG/EAGE overthrust model simulation \cite{aminzadeh1994seg}. The dimension of the model is $5\,\mathrm{km} \times 10\,\mathrm{km} \times 10\,\mathrm{km}$ and is discretized on a $25\,\mathrm{m} \times 25\,\mathrm{m} \times 25\,\mathrm{m}$ grid. The simulated data contains $201 \times 201$ receivers sampled at 50 m and $101 \times 101$ sources sampled at 100 m. 
We apply the Fourier transform along the time domain and extract a frequency slice at 10 Hz as shown in Figure~\ref{fig:true}, 
which is a 4D object (source-x, source-y, receiver-x and receiver-y). 
We eliminate 80\% of the sources and add large sparse outliers from the random gaussian distribution 
$\mathcal{N}(0, a_i\max(X_{s_i}))$ (mean zero and variance on the order of the largest value in that particular source). 
The 10 generated values with the highest magnitudes are kept, and these are randomly added to observations in the remaining sources (Figure~\ref{fig:interdeno_data}). 
The largest value of our dataset is approximately 40, while the smallest is close to zero. 
Thus, we are essentially increasing/decreasing 1\% of the entries by several orders of magnitude, which contaminates the data significantly, especially 
 if the original entry was nearly $0$. 
 For all low-rank completion and denoising, we let $a_i = 10^{-1}$. 
 The objective is to recover missing sources and eliminate noise from observed data.
 We use a rank of $k=75$ for the formulation (that is, $L\in\mathbb{C}^{n\times 75}$ and similarly for $R$), 
 and run all algorithms for 150 iterations, using a fixed computational budget. 
 We perform three experiments on the same dataset: 1) De-noising only (Figure~\ref{fig:deno_data}); 2) Interpolation only (Figure~\ref{fig:inter_data}); and 3) Combined Interpolation and De-noising (Figure~\ref{fig:interdeno_data}). Since we have ground truth, we pick $\sigma$ to be the exact difference between generated noisy data and the true data; $\sigma$ for the $l_0$ norm is a cardinality measure, so it is set to number of noisy points added. 

% Figures~\ref{fig:fig4} (a,c) show the interpolation and denoising results for a monochromatic frequency slice using SPGLR and the proposed formulation. We can clearly see that for both the formulations, we are able to reconstruct the coherent energy. Since the subsampling ratio is very high, we expect to have low-amplitude coherent noise to be in residual (Figures~\ref{fig:fig4} (b,d)). To show the details in the reconstructed data, we further plot the zoomed-in sections (Figures~\ref{fig:fig5}). Also, we find that the propose formulation gets an improved quality of interpolation results compared to the SPGLR formulation in addition to handling non-smooth penalties effectively.  Finally, Table~\ref{tbl:data_4d} summarizes the comparison of reconstruction using the different penalty norms in the propose formulation.
\begin{figure}
\centering
\subfigure[(source-x, source-y)]{\label{fig:svalsa}\includegraphics[scale=0.5]{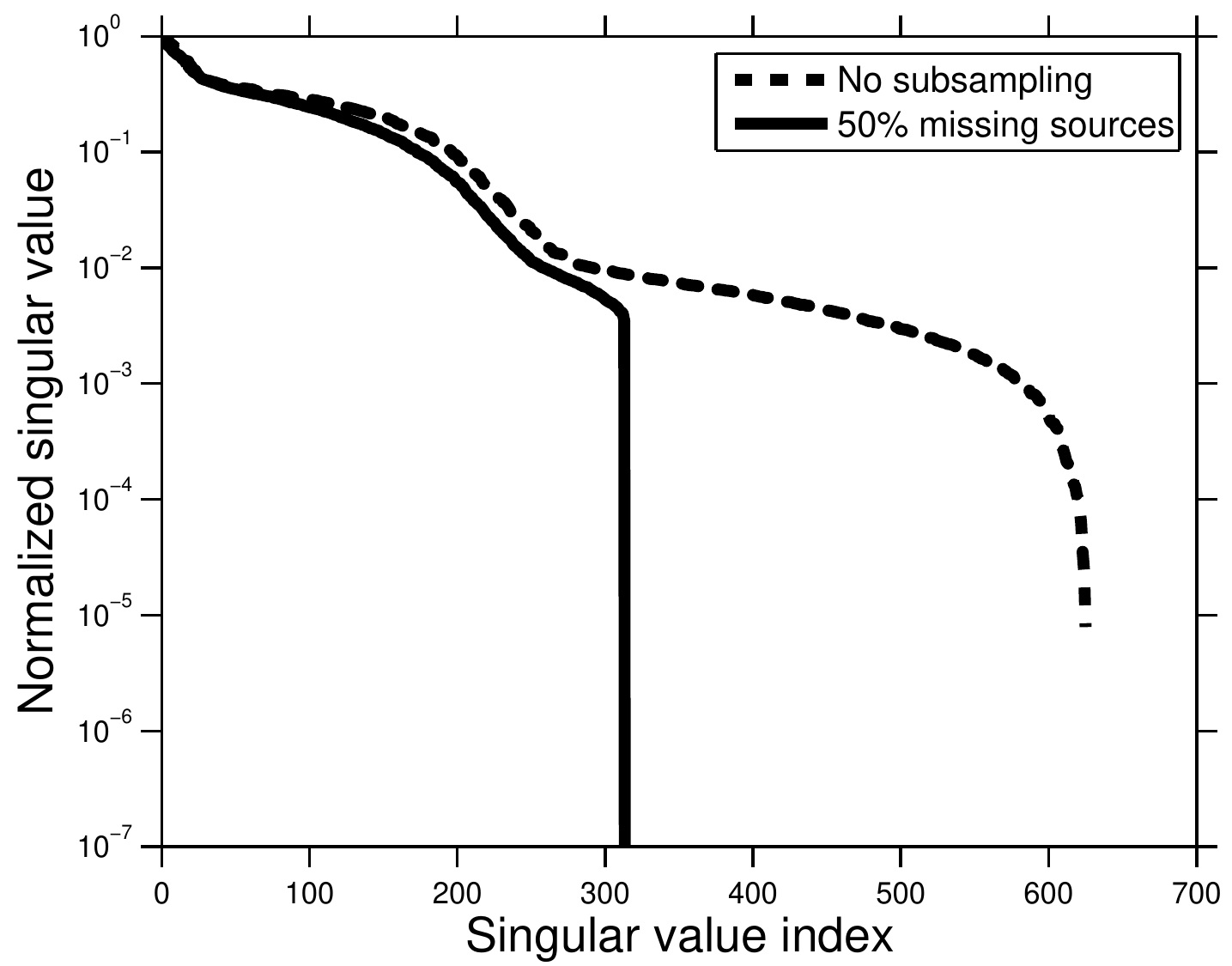} }
\subfigure[(source-x, receiver-x)]{\label{fig:svalsb}\includegraphics[scale=0.5]{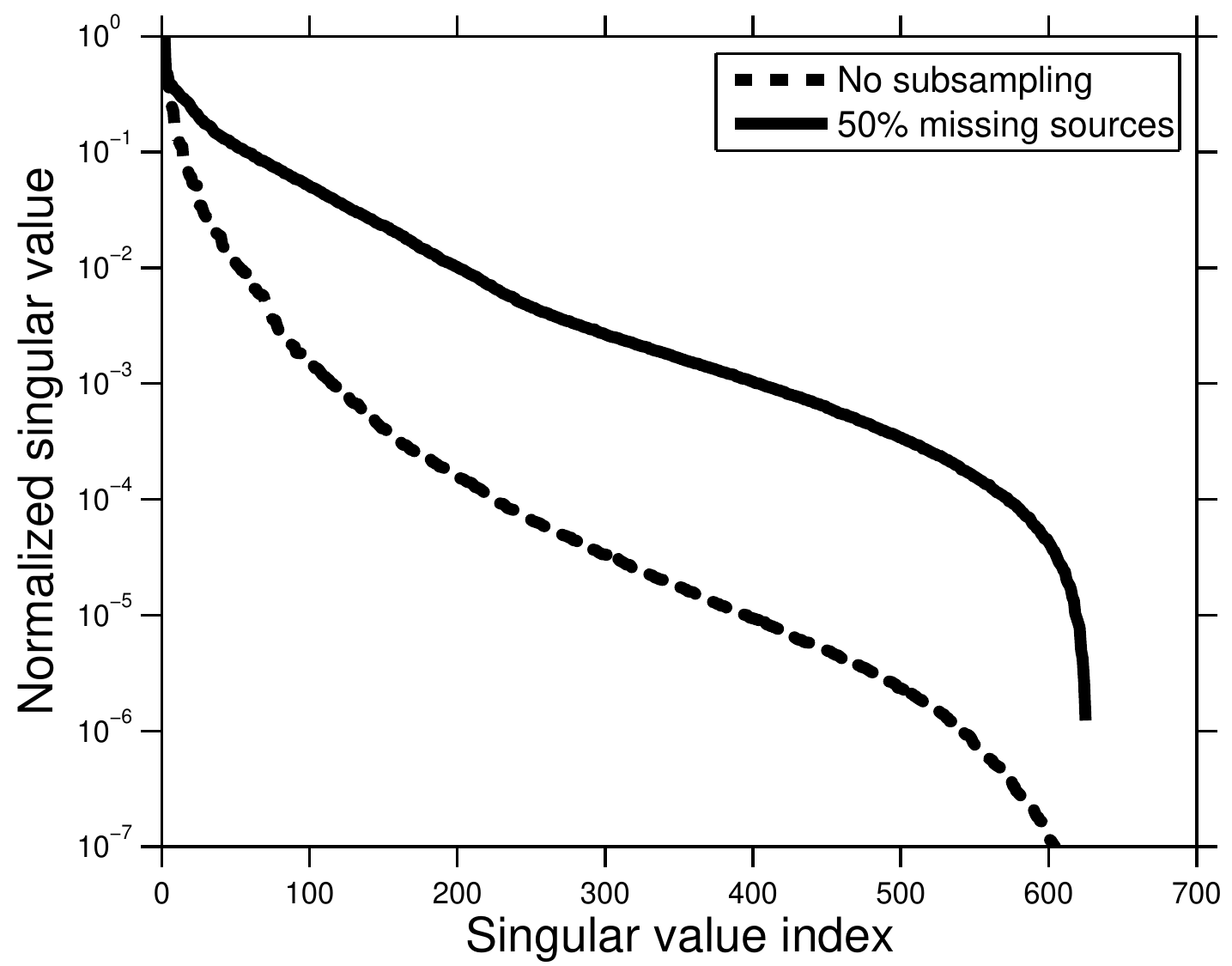}}
\caption{Normalized Singular value decay for full data and 50\% missing sources with two different matricizations. (Source: \cite{kadu2018EAGEdfwi}).}
\label{fig:svals}
\end{figure}

\begin{figure}
\centering
\subfigure[Full (src-x, src-y)]{\label{fig:sxsy}\includegraphics[scale=0.2]{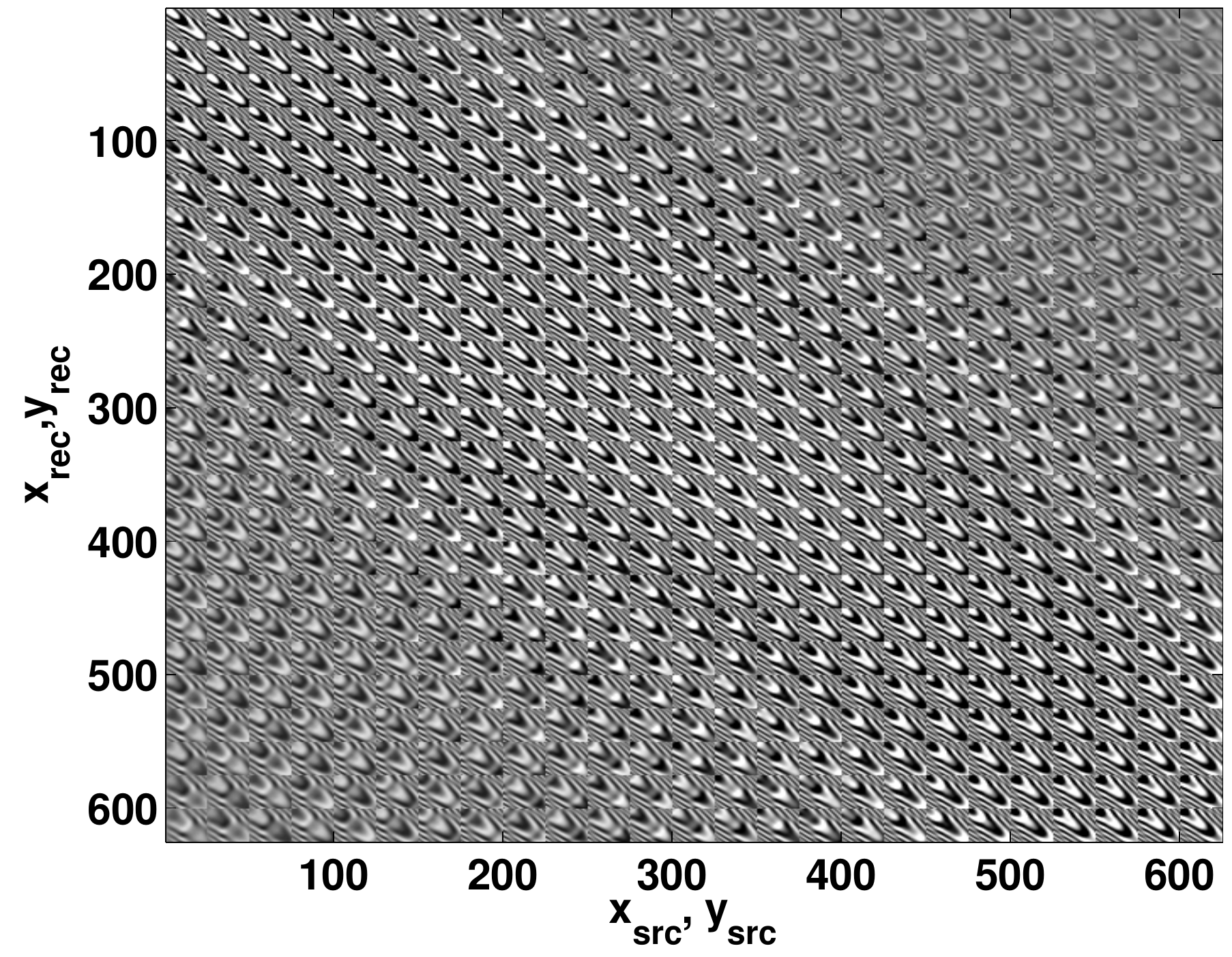} }
\subfigure[Subsampled (src-x, src-y)]{\label{fig:sxsy_sub}\includegraphics[scale=0.2]{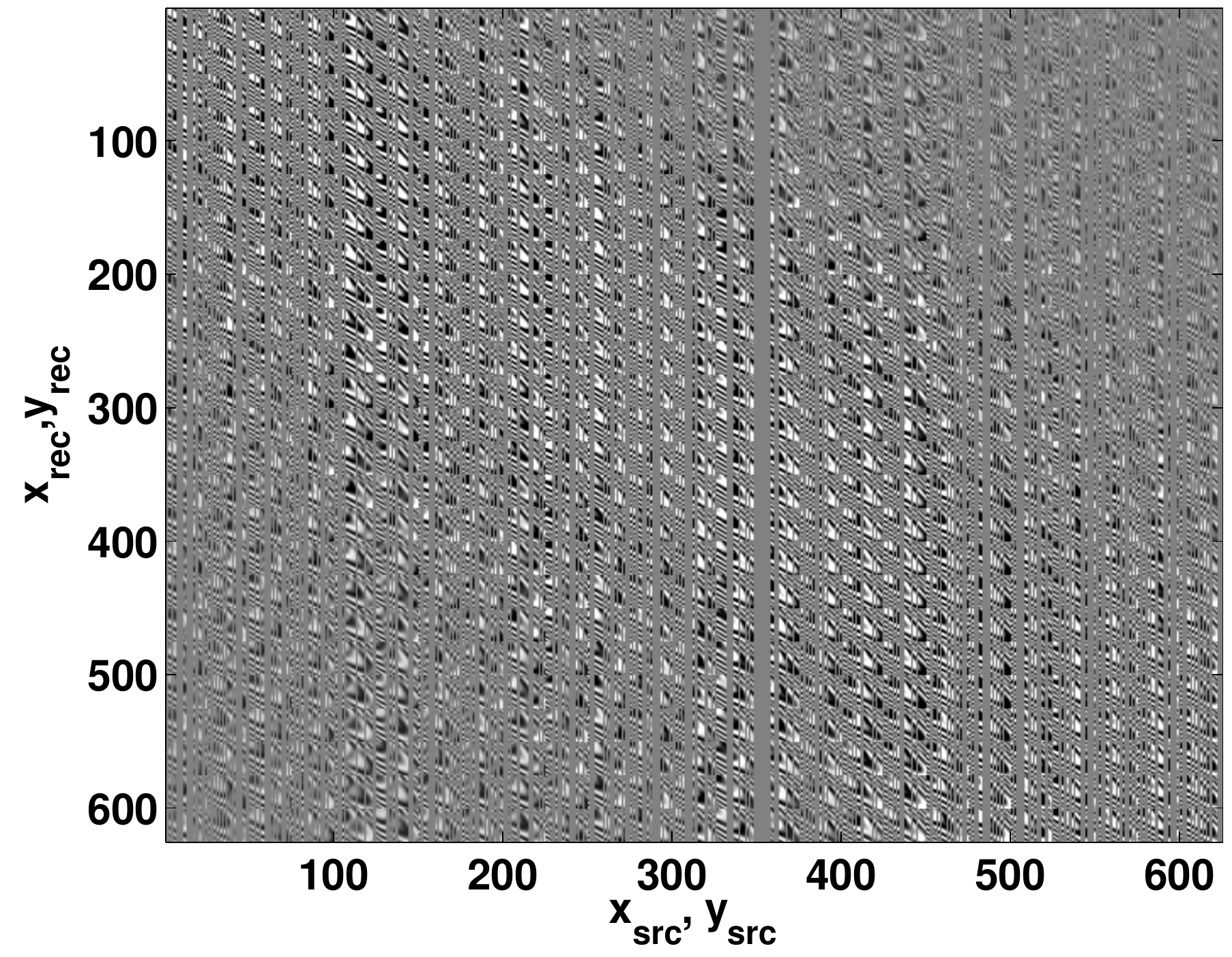} }
\subfigure[Full (src-x, rec-x)]{\label{fig:sxrx}\includegraphics[scale=0.2]{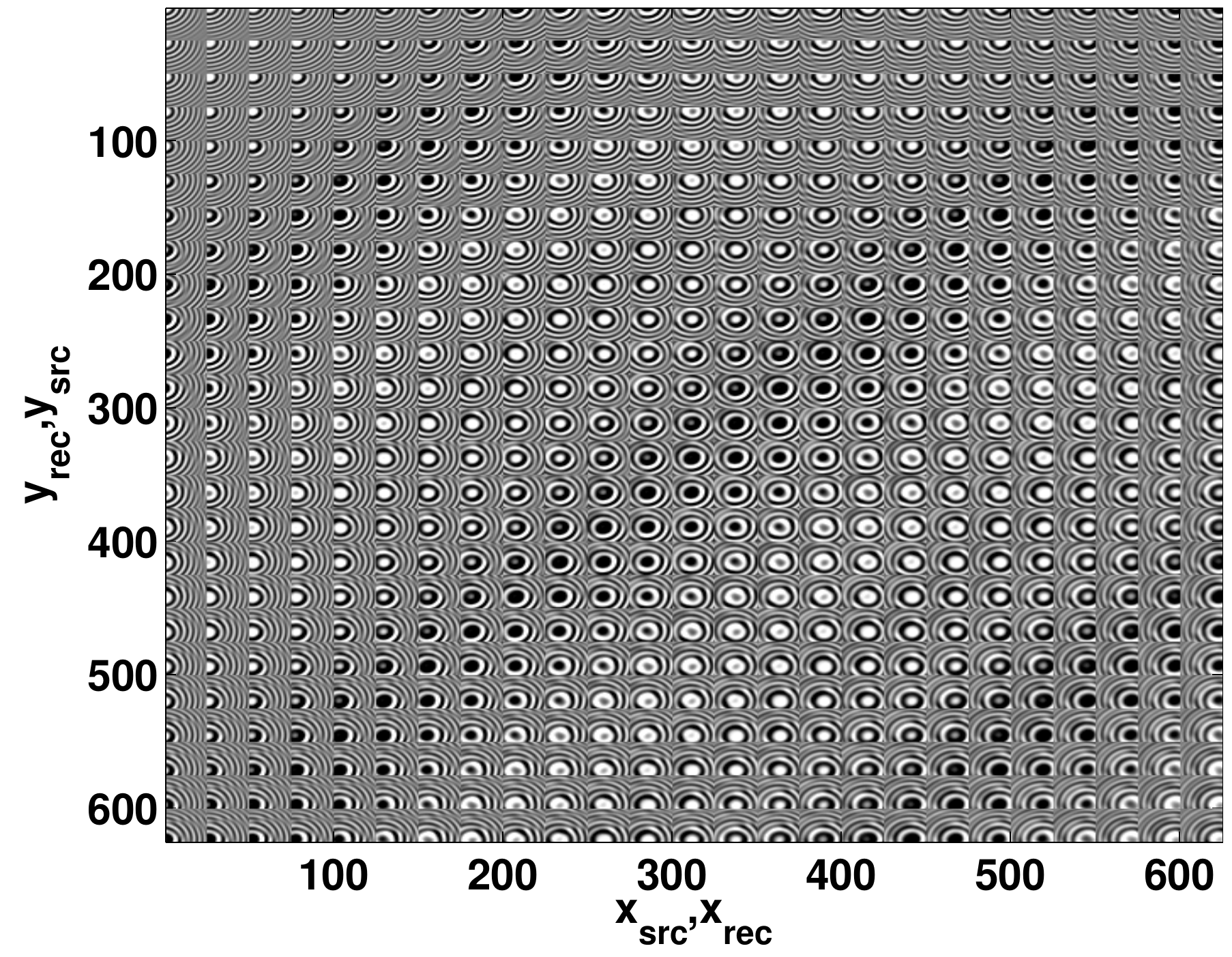}}
\subfigure[Subsampled (src-x, rec-x)]{\label{fig:sxrx_sub}\includegraphics[scale=0.2]{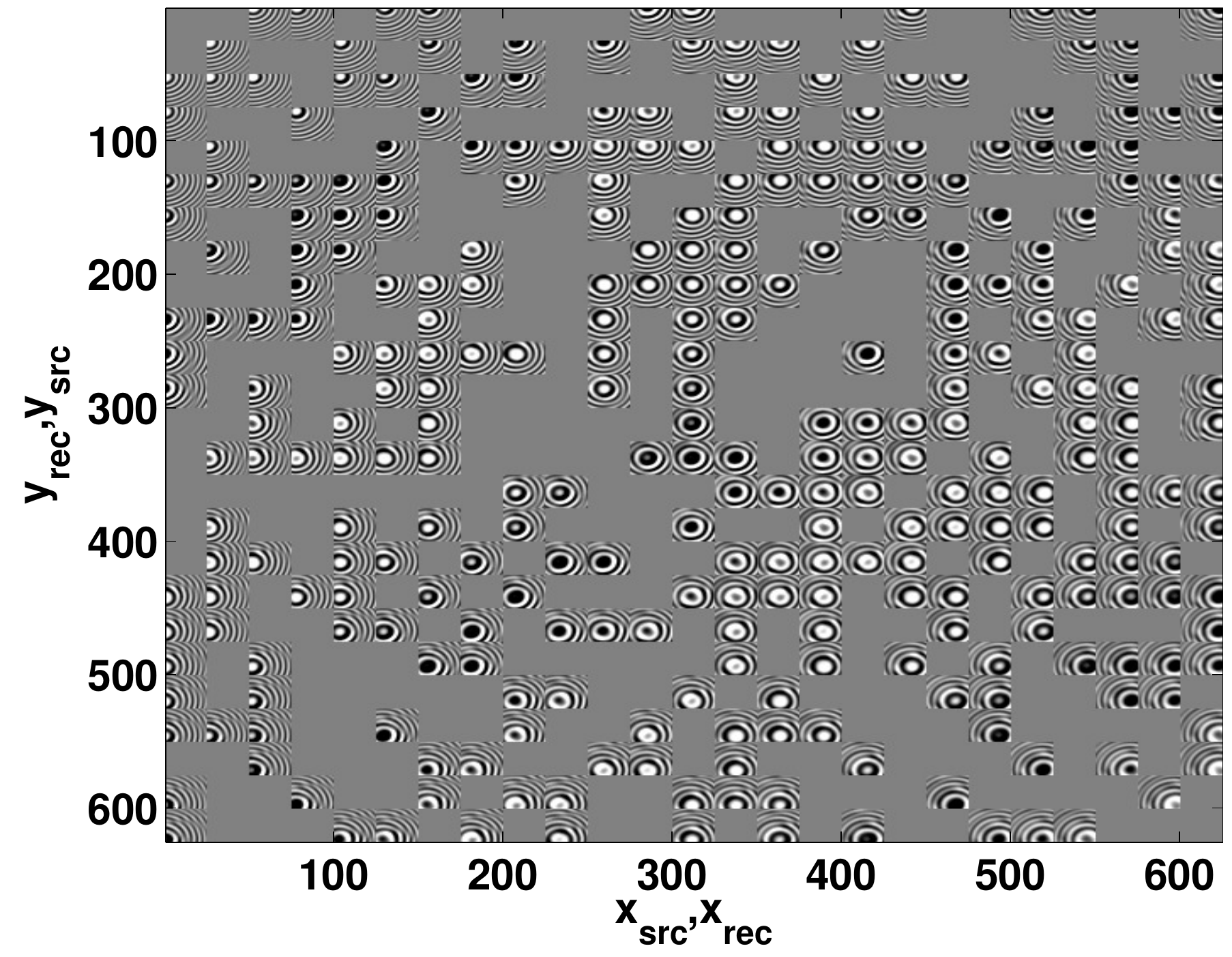} }
\caption{Full and subsampled matricizations used in low-rank completion (Source: \cite{kumar2015efficient}).}
\label{fig:srcxsrcymatricization}
\end{figure}

\subsection{Results}
Tables \ref{tab:deno}-\ref{tab:interdeno} display SNR values for different algorithms and formulations for the three types of experiments, 
and Figures \ref{fig:deno}-\ref{fig:interdeno} display the results for a randomly selected number of sources for the three experiments. 
Even a small number of outliers can greatly impact the quality of the low-rank de-noising and interpolation for the standard, smoothly residual-constrained algorithms. 
The de-noising only results (Figure \ref{fig:deno}, Table \ref{tab:deno}) show that all methods perform well when all sources are available. 
%The $\ell_0$-norm does better for de-noising only, but with all of the data present, the other formulations can rely on the low-rank structure of the data to fill in missing gaps. 
The interpolation only results (Figure \ref{fig:inter}, Table \ref{tab:inter}) show that all constraints perform well in interpolating the missing data. 
This makes sense, as all algorithms will simply favor the low-rank nature of the data. 
%The $\ell_0$-norm is only comparable to other algorithms and formulations for the interpolation only dataset. 
However, the combined de-noising and interpolation dataset shows that the $\ell_0$ norm approach does far better than any smooth norm in comparable time. 
Table \ref{tab:interdeno} shows that when data for similar sources is absent/not observed, the smoothly-constrained formulations fail completely.
When noise is added to the low-amplitude section of the observed data, 
the smoothly-constrained norms fail drastically, while the $\ell_0$ norm can effectively remove the errors. 
This is starkly evident in Figures \ref{fig:interdeno_spglr}-\ref{fig:interdeno_l0}, where all except Figure \ref{fig:interdeno_l0} are essentially noise;
the result is supported by the SNR values in Table \ref{tab:interdeno}. While Figures \ref{fig:interdeno_spglr}-\ref{fig:interdeno_l0} can mostly capture the structure of the data where there were nonzero values (ie where the seismic wave is observed in the upper left corner of each source), only the $\ell_0$ norm can capture the areas of lower energy data.

  \begin{table}
  \centering
  \caption{4D De-noising results for SPGLR and Algorithm \ref{alg:block-descent} for selected $\ell_p$ norms.}
  \label{tab:deno}
    \begin{tabular}{|c|c|c|c|}
      \hline
      \multicolumn{4}{|c|}{4D Monochromatic De-noising} \\
      \hline
      Method/Norm & SNR & SNR-W & Time (s)\\
      \hline
      $\ell_2$ with SPGLR & 11.7489 &  - & 16530   \\
      $l_2$ with~Alg.\ref{alg:block-descent} &   11.7463 &-2.3338 &   9430    \\\hline
      $l_1$ with~Alg.\ref{alg:block-descent} & 11.7638 &-2.3063 &    11546    \\
      $l_\infty$ with~Alg.\ref{alg:block-descent}& 11.7456  & -2.3338&  12108    \\
      $l_0$ with~Alg.\ref{alg:block-descent} &    17.9595& 48.8607&    11569\\
      \hline
    \end{tabular}
    \end{table}
  \begin{table}
  \centering
  \caption{4D Interpolation results for SPGLR and Algorithm \ref{alg:block-descent} for selected $\ell_p$ norms.}
  \label{tab:inter}
    \begin{tabular}{|c|c|c|c|}
      \hline
      \multicolumn{4}{|c|}{4D Monochromatic Interpolation} \\
      \hline
      Method/Norm & SNR &SNR-W & Time (s)\\
      \hline
      $\ell_2$ with SPGLR & 16.3976&  - & 5817    \\
      $l_2$ with~Alg.\ref{alg:block-descent} &   16.0629& 16.5424 &    7526     \\\hline
      $l_1$ with~Alg.\ref{alg:block-descent} & 16.0692&16.5491  &   7996   \\
      $l_\infty$ with~Alg.\ref{alg:block-descent}& 16.0627&16.5423  &  8119    \\
      $l_0$ with~Alg.\ref{alg:block-descent} &    16.0096&16.4728 &    6848 \\
      \hline
    \end{tabular}
    \end{table}

  \begin{table}
  \centering
  \caption{4D Combined De-noising and Interpolation results for SPGLR and Algorithm \ref{alg:block-descent} for selected $\ell_p$ norms.}
  \label{tab:interdeno}
    \begin{tabular}{|c|c|c|c|}
      \hline
      \multicolumn{4}{|c|}{4D Monochromatic De-noising \& Interpolation} \\
      \hline
      Method/Norm & SNR &SNR-W &  Time (s)\\
      \hline
      $\ell_2$ with SPGLR & -3.2906 & -   & 8712    \\
      $l_2$ with~Alg.\ref{alg:block-descent} &   0.9185& -0.3321  &    6802     \\\hline
      $l_1$ with~Alg.\ref{alg:block-descent} & 0.9193& -0.3235  &    8068    \\
      $l_\infty$ with~Alg.\ref{alg:block-descent}& 0.9185& -0.3321  &  8117    \\
      $l_0$ with~Alg.\ref{alg:block-descent} &    16.0655& 16.5445 &    6893 \\
      \hline
    \end{tabular}
    \end{table}

% The first argument of the \texttt{multiplot} command specifies the
% number of plots per row.

% \subsection{Tables}

\section{Conclusions}
We proposed a new approach for level-set formulations, including basis pursuit denoise and residual-constrained low-rank formulations. 
The approach is easily adapted to a variety of nonsmooth and nonconvex data constraints. 
The resulting problems are solved using Algorithm~\ref{alg:varpro} and \ref{alg:block-descent}; which require 
only that  the penalty $\rho$ has an efficient projector. 
The algorithms are simple, scalable, and efficient. 
Sparse curvelet de-noising and low-rank interpolation of a monochromatic slice from the 4D seismic data volumes demonstrate the potential of the approach. 

A particular quality of the seismic denoising and interpolation problem is that the amplitudes of the signal have significant spatial variation.  
The error in the data is a much larger problem for low-amplitude data. This quality makes it very difficult to obtain reasonable results using 
Gaussian misfits and constraints. Nonsmooth exact formulations (including $\ell_1$ and particularly $\ell_0$) appear to be extremely well-suited
for this magnified heteroscedastic issue. 

\section{Acknowledgements}
The authors acknowledge support from the Department of Energy Computational Science Graduate Fellowship, which is provided under grant number DE-FG02-97ER25308,
and the Washington Research Foundation Data Science Professorship.

\begin{figure*}
\centering     %%% not \center
\subfigure[Fully sampled monochromatic slize at 10 Hz.]{\label{fig:true}\includegraphics[scale=0.14]{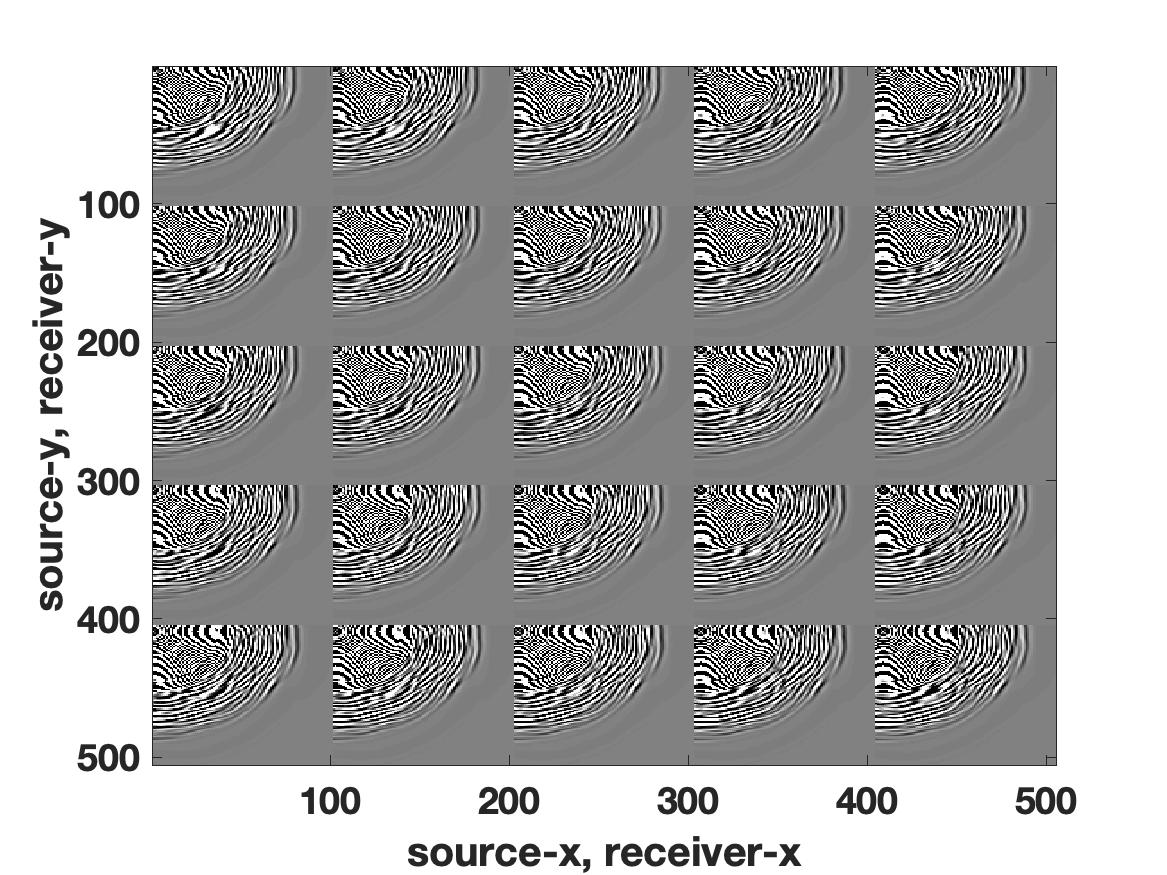}} %\qquad
\subfigure[Noisy data alone (binary). Sparse noise was added by keeping the top 10 entries generated from a normal distribution with mean zero and variance $0.1\max(X_{s_{i}})$]{\label{fig:deno_noise}\includegraphics[scale=0.14]{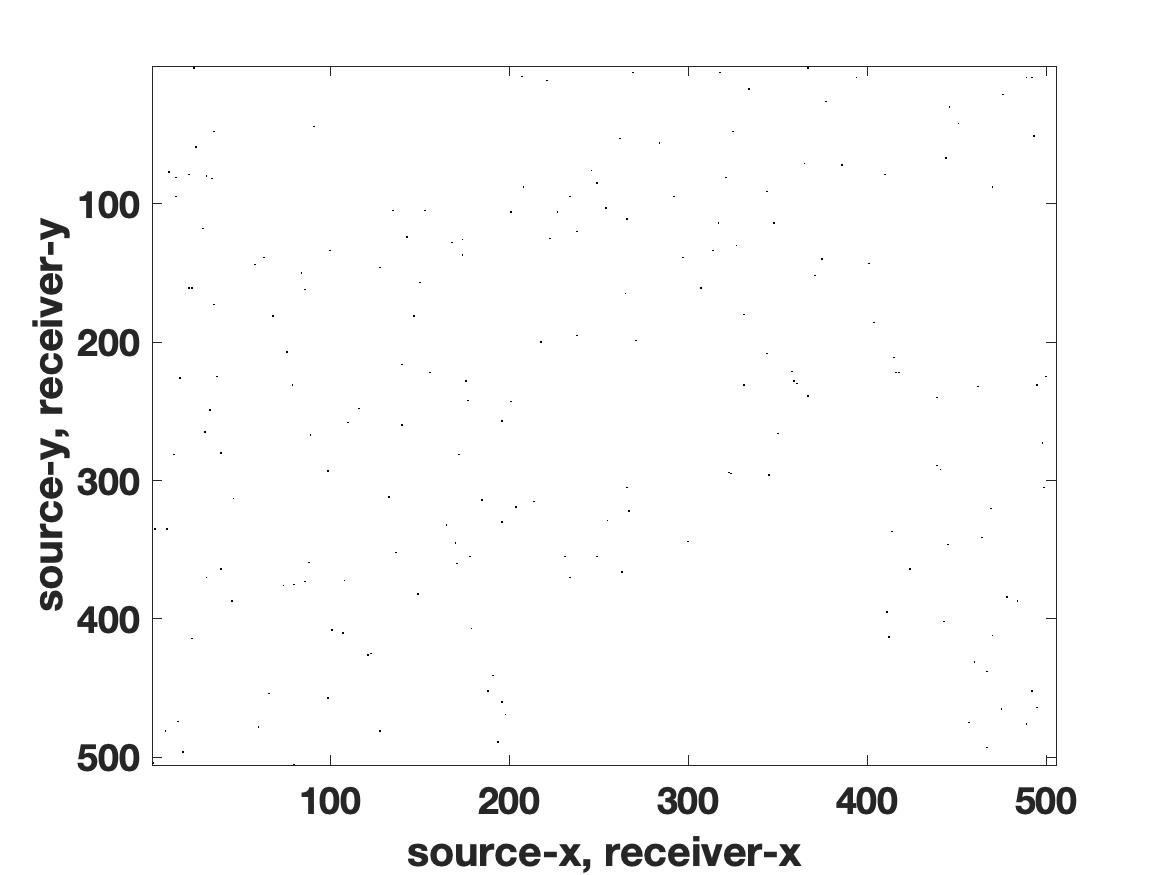}}
\subfigure[Observed noisy data.]{\label{fig:deno_data}\includegraphics[scale=0.14]{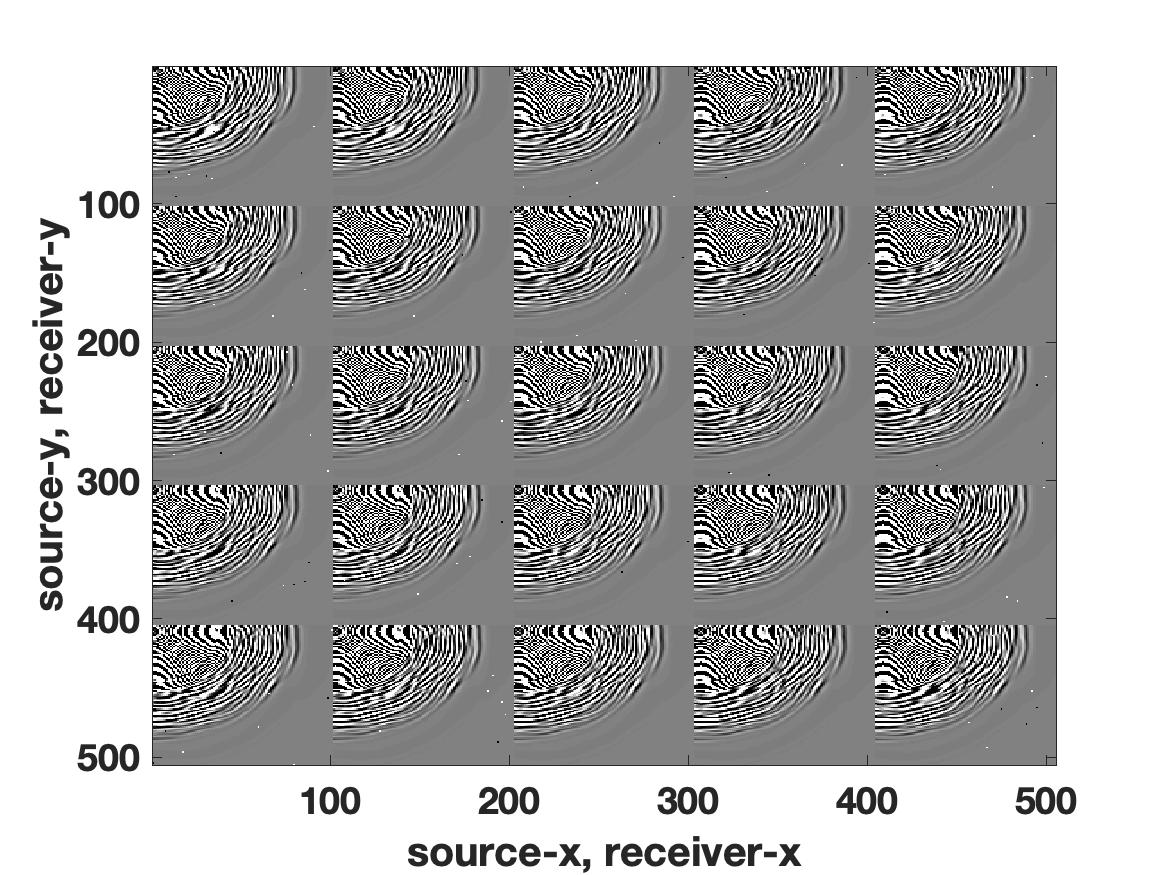}}
\subfigure[Subsampled noiseless data. We omitted 80\% of sources.]{\label{fig:inter_data}\includegraphics[scale=0.14]{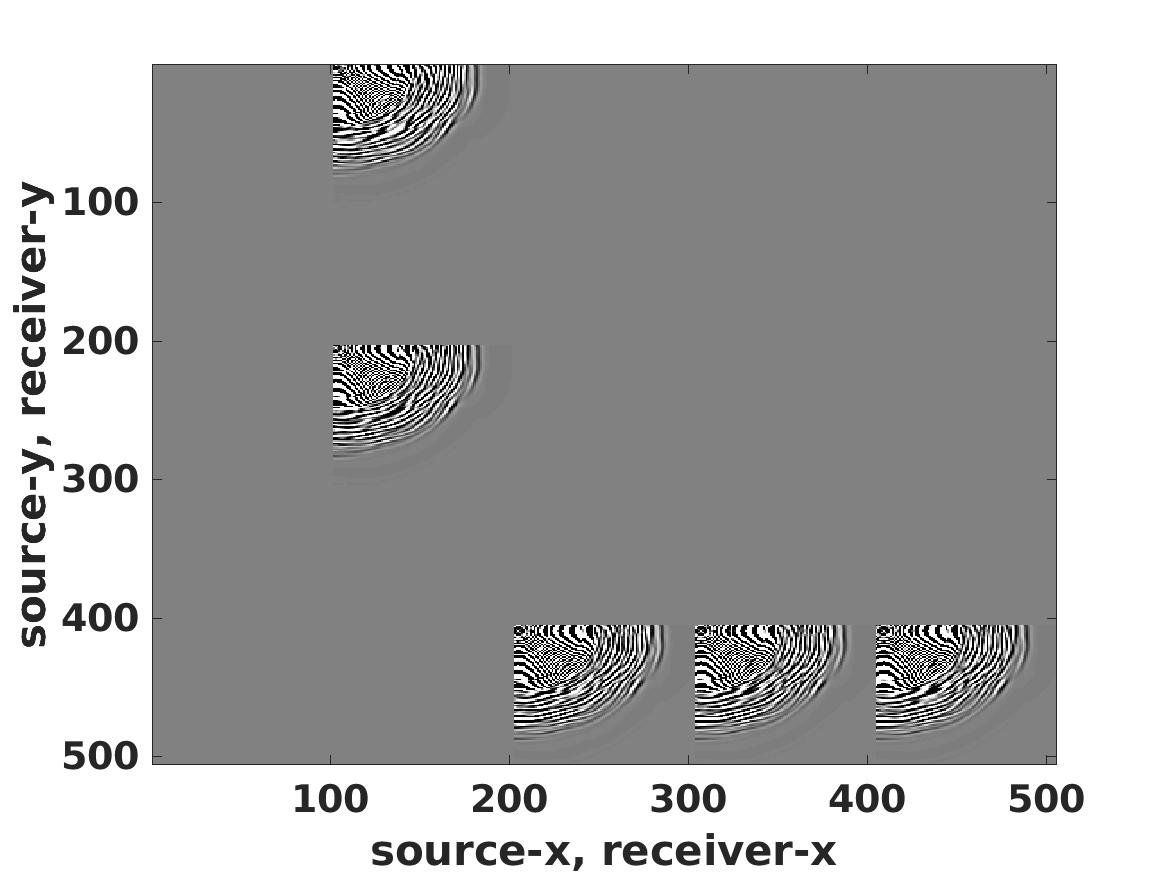}}
\subfigure[Subsampled and noise, with noise only present (binary).]{\label{fig:interdeno_noise}\includegraphics[scale=0.14]{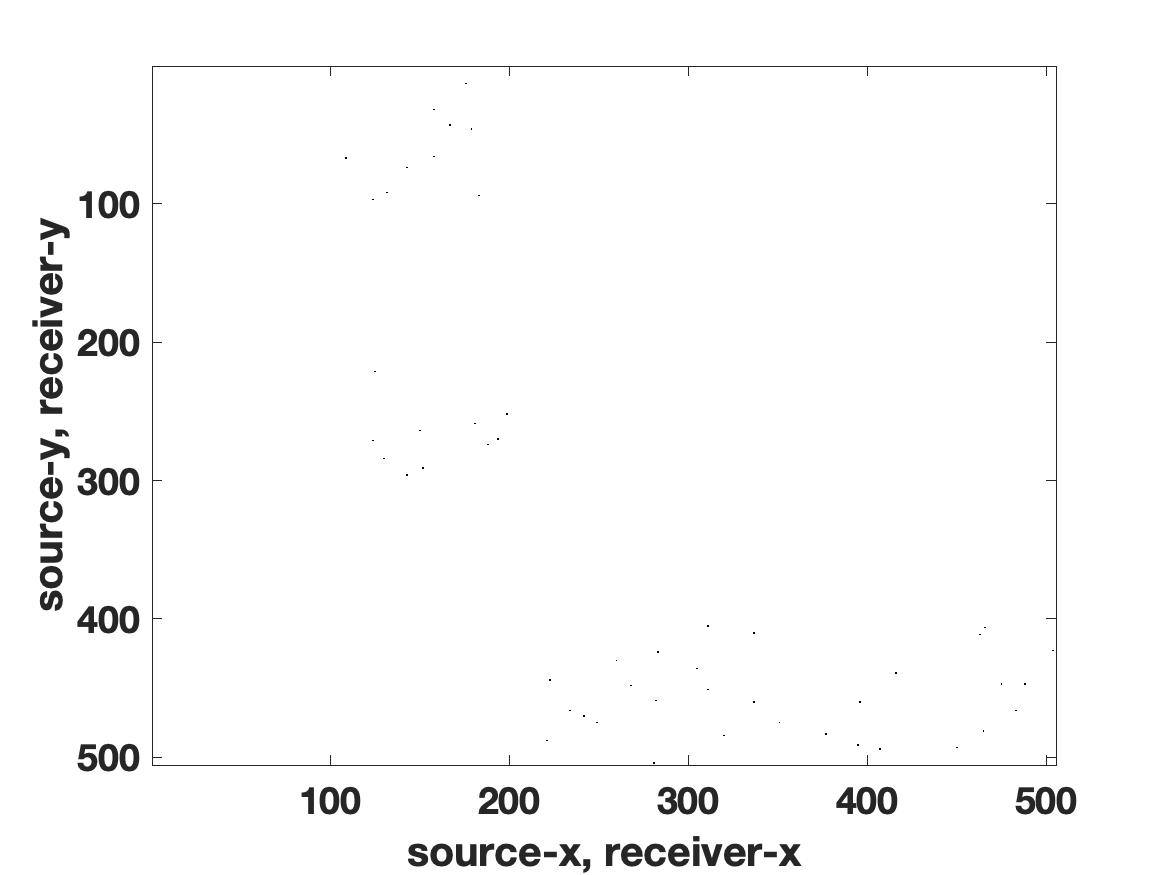}}
\subfigure[Subsampled and noisy data. We again omitted 80\% of sources and added the noise described above to the rest of the sources.]{\label{fig:interdeno_data}\includegraphics[scale=0.14]{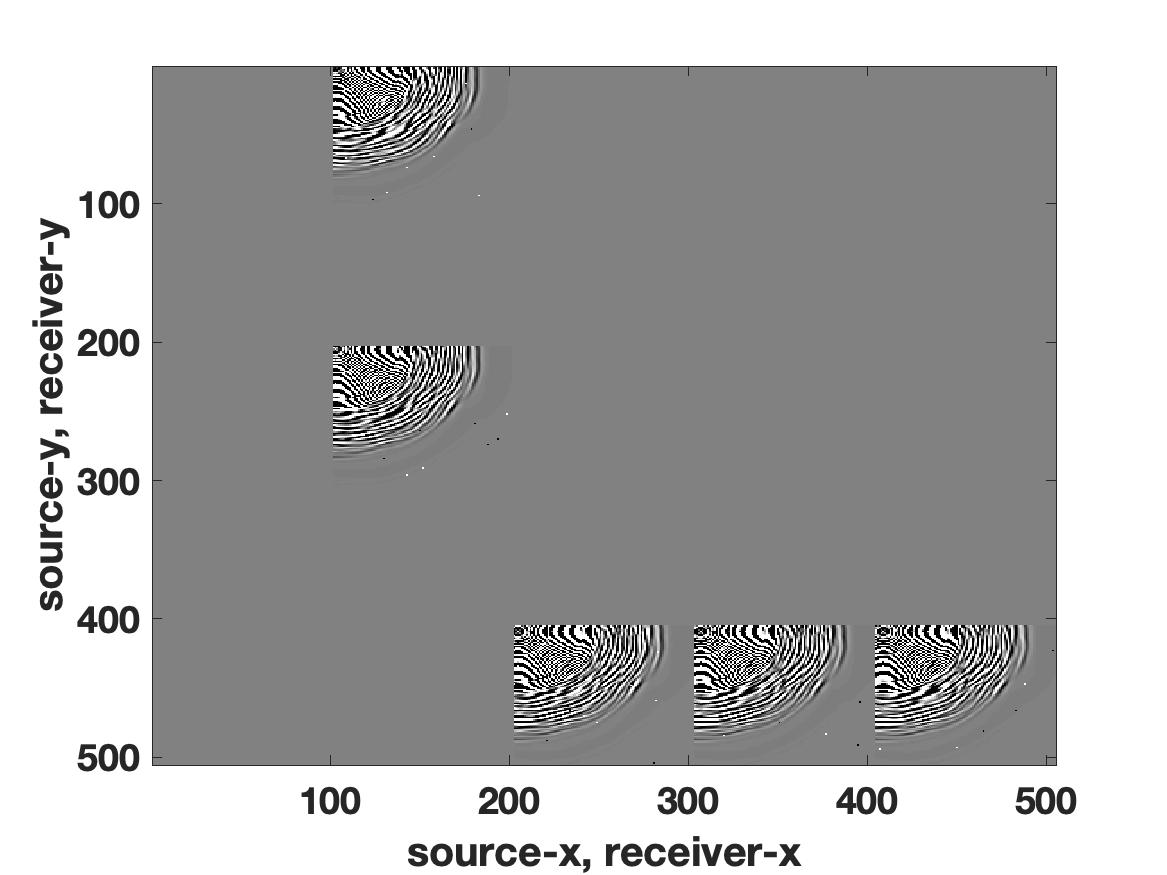}}
\caption{True data and three different experiments for testing our completeness algorithm.}
\label{fig:data}
\end{figure*}
\begin{figure*}
\centering     %%% not \center
\subfigure[SPGLR]{\label{fig:deno_spglr}\includegraphics[scale=0.14]{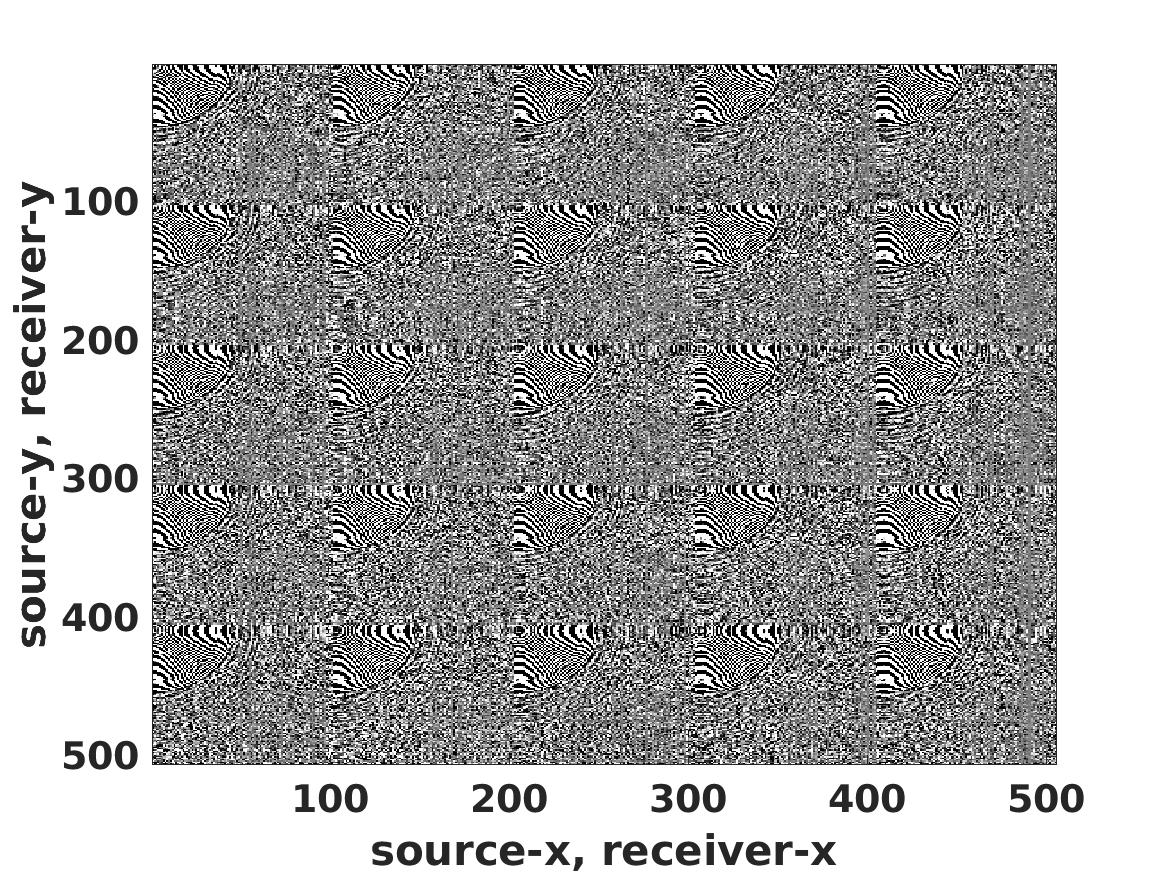}} %\qquad
\subfigure[$l_2$]{\label{fig:deno_l2}\includegraphics[scale=0.14]{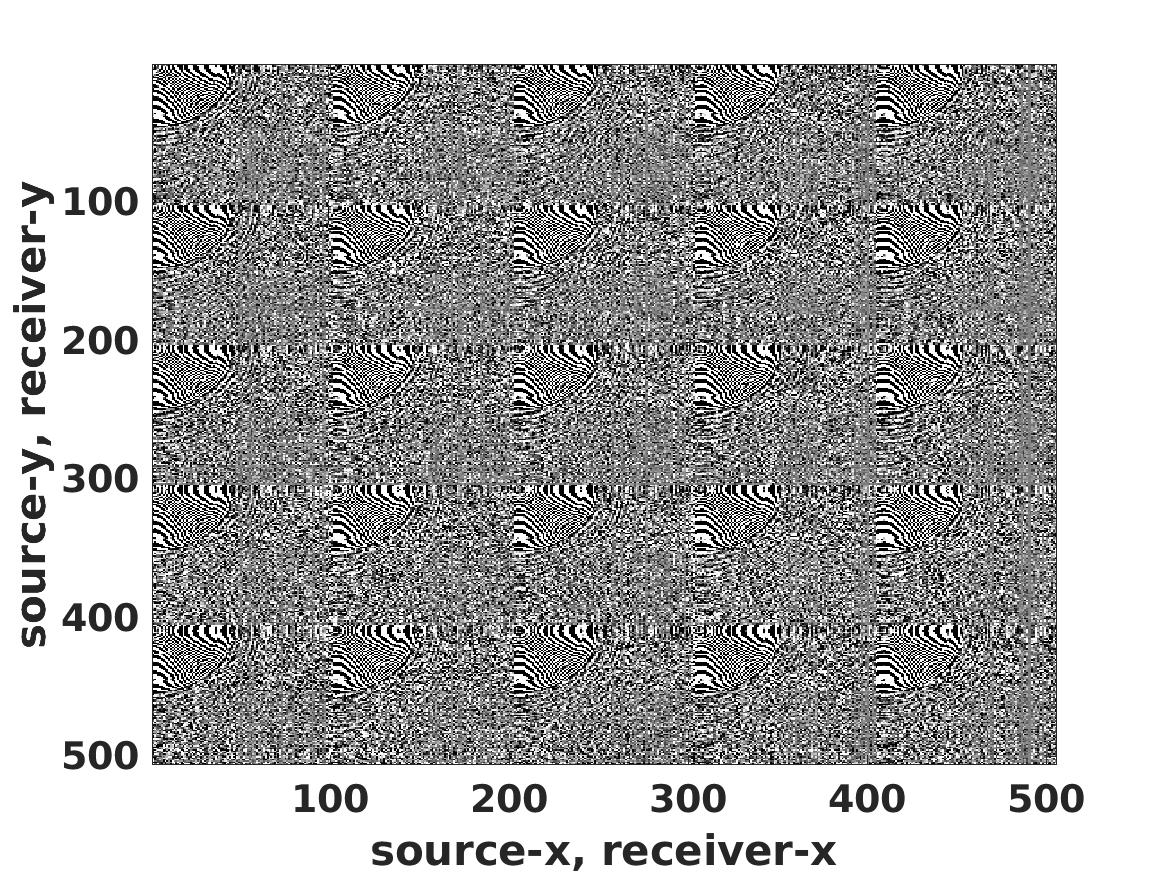}} %\qquad
\subfigure[$l_1$]{\label{fig:deno_l1}\includegraphics[scale=0.14]{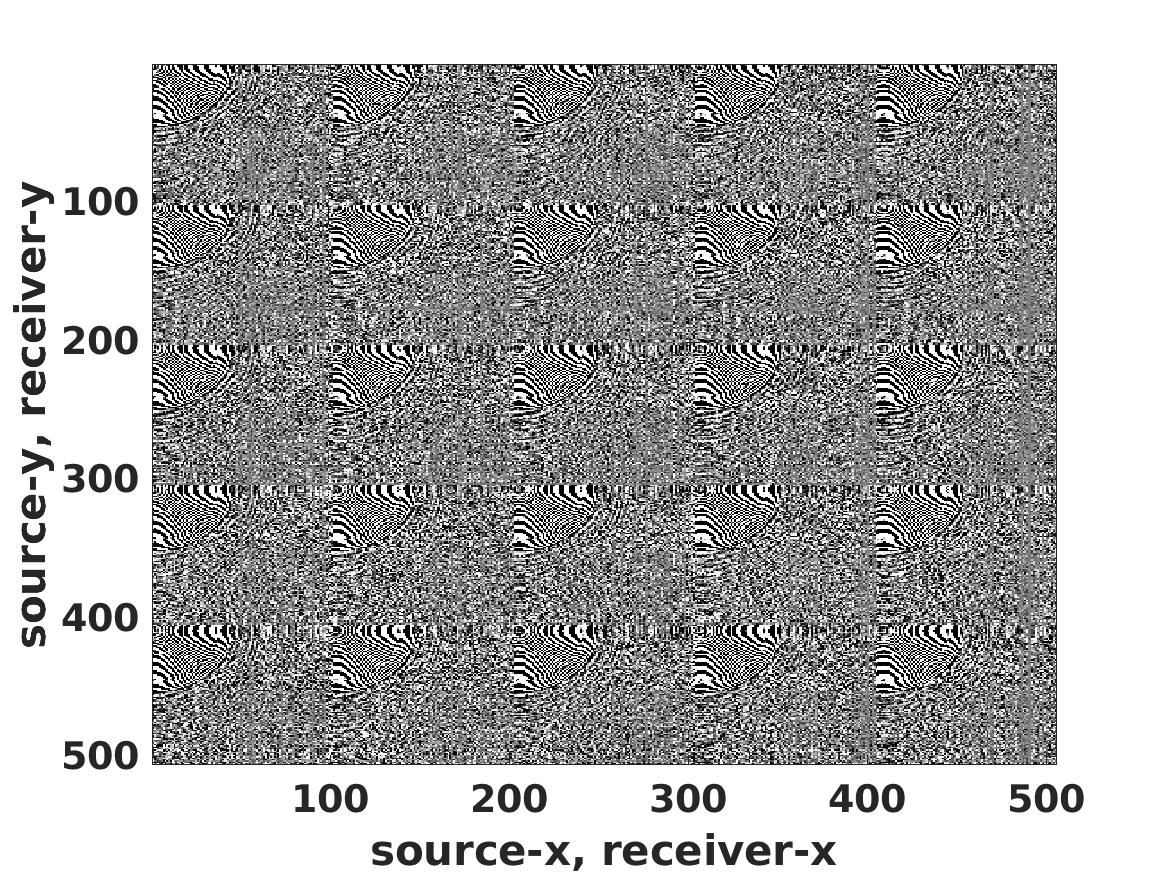}} %\qquad
\subfigure[$l_\infty$]{\label{fig:deno_linf}\includegraphics[scale=0.14]{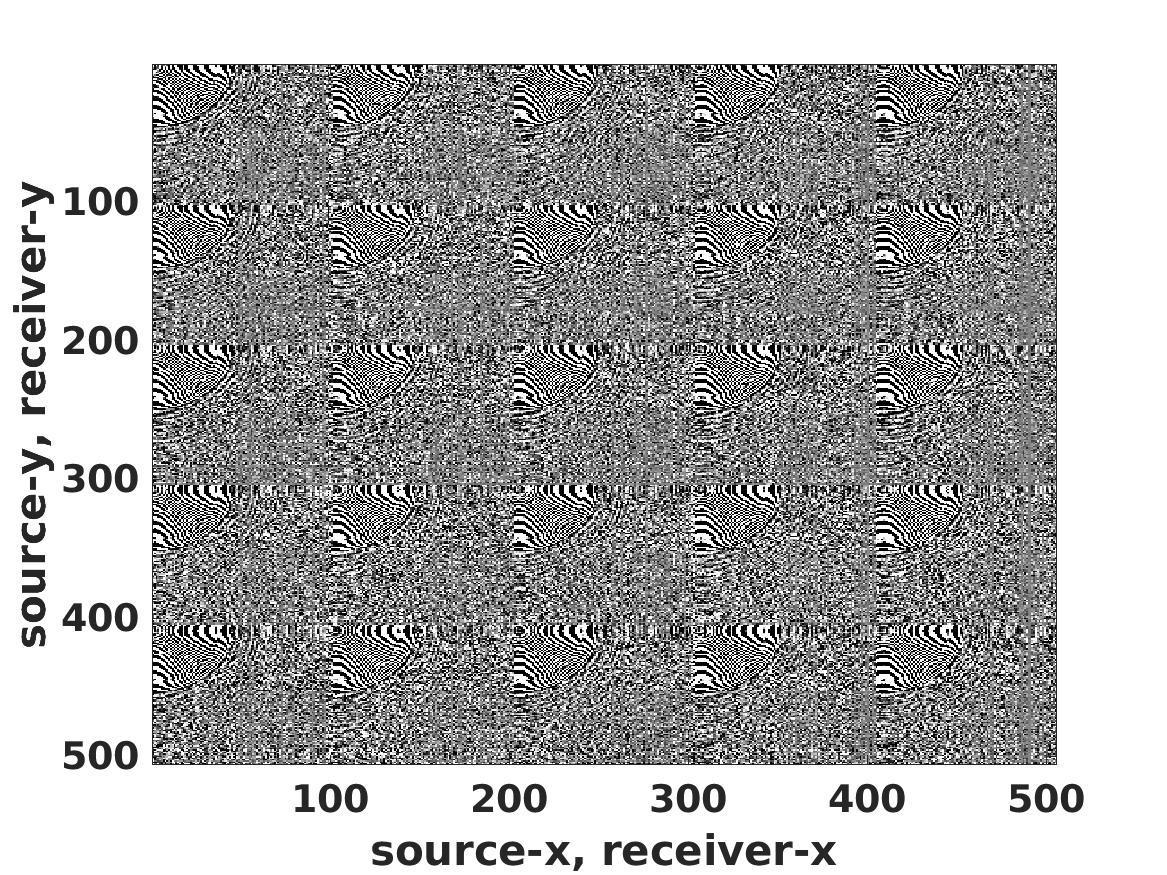}} %\qquad
\subfigure[$l_0$]{\label{fig:deno_l0}\includegraphics[scale=0.14]{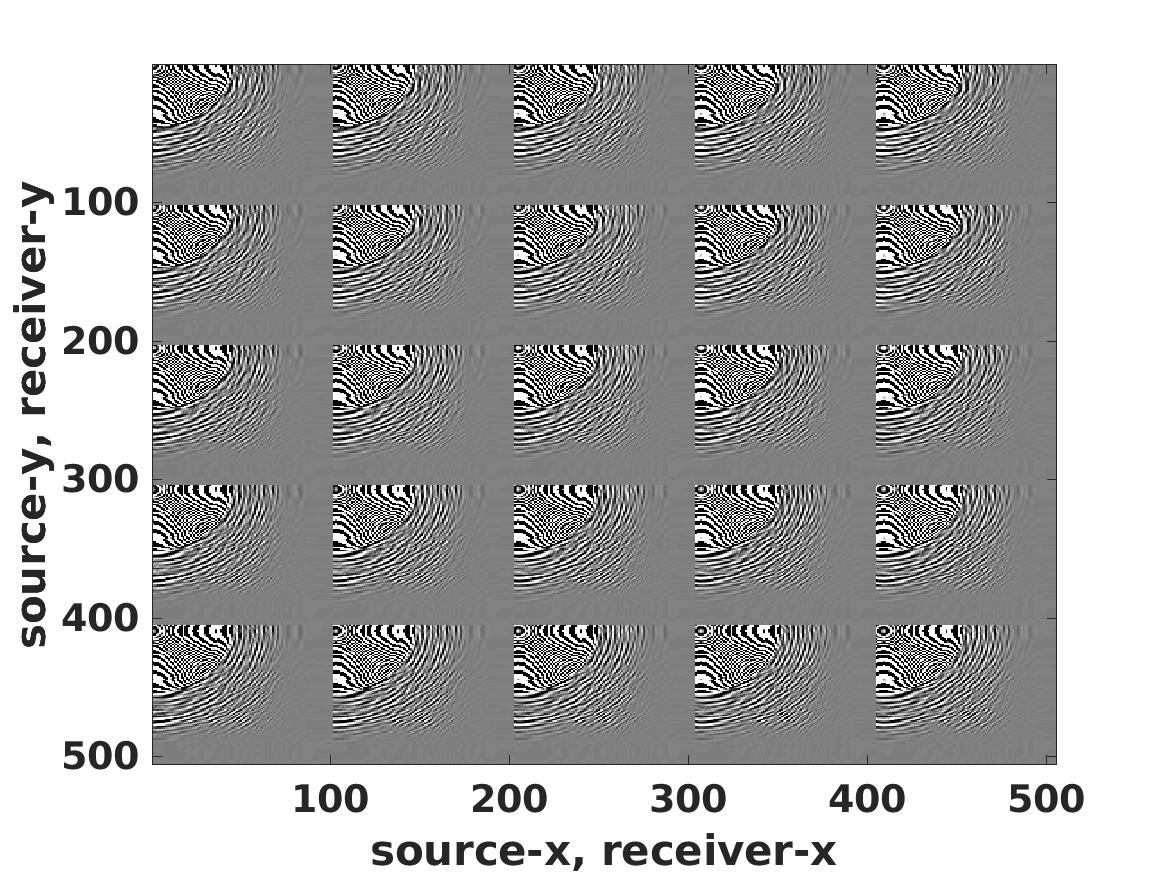}} %\qquad
\caption{Denoising-only results.}
\label{fig:deno}
\end{figure*}
\begin{figure*}
\centering     %%% not \center
\subfigure[SPGLR]{\label{fig:inter_spglr}\includegraphics[scale=0.14]{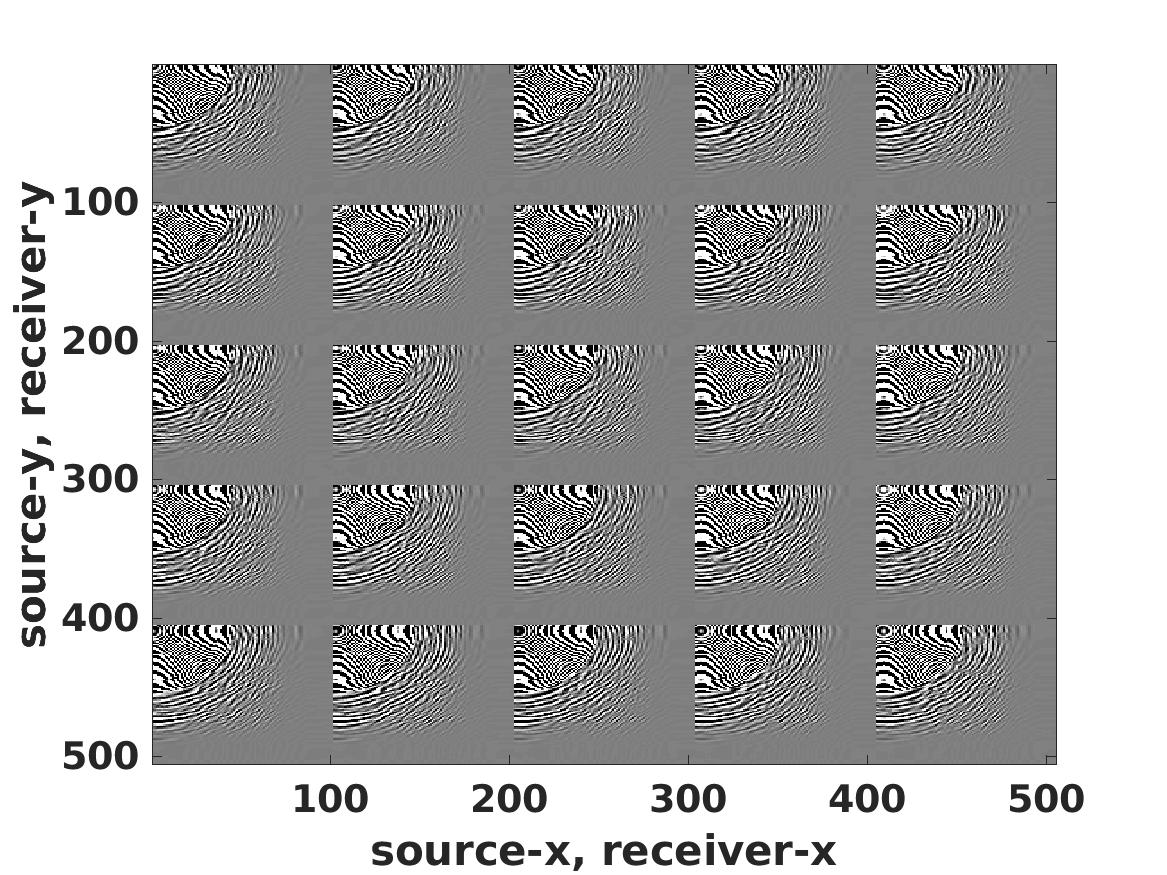}} %\qquad
\subfigure[$l_2$]{\label{fig:inter_l2}\includegraphics[scale=0.14]{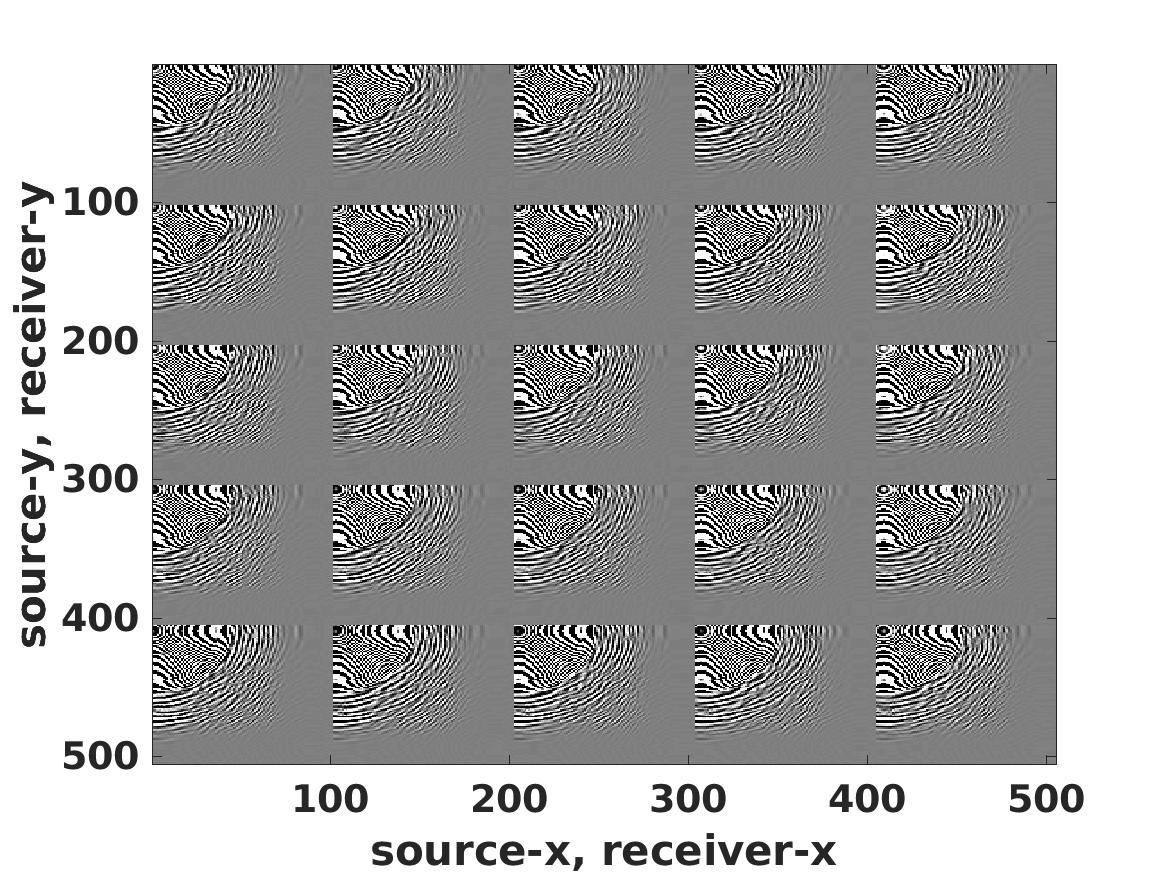}} %\qquad
\subfigure[$l_1$]{\label{fig:inter_l1}\includegraphics[scale=0.14]{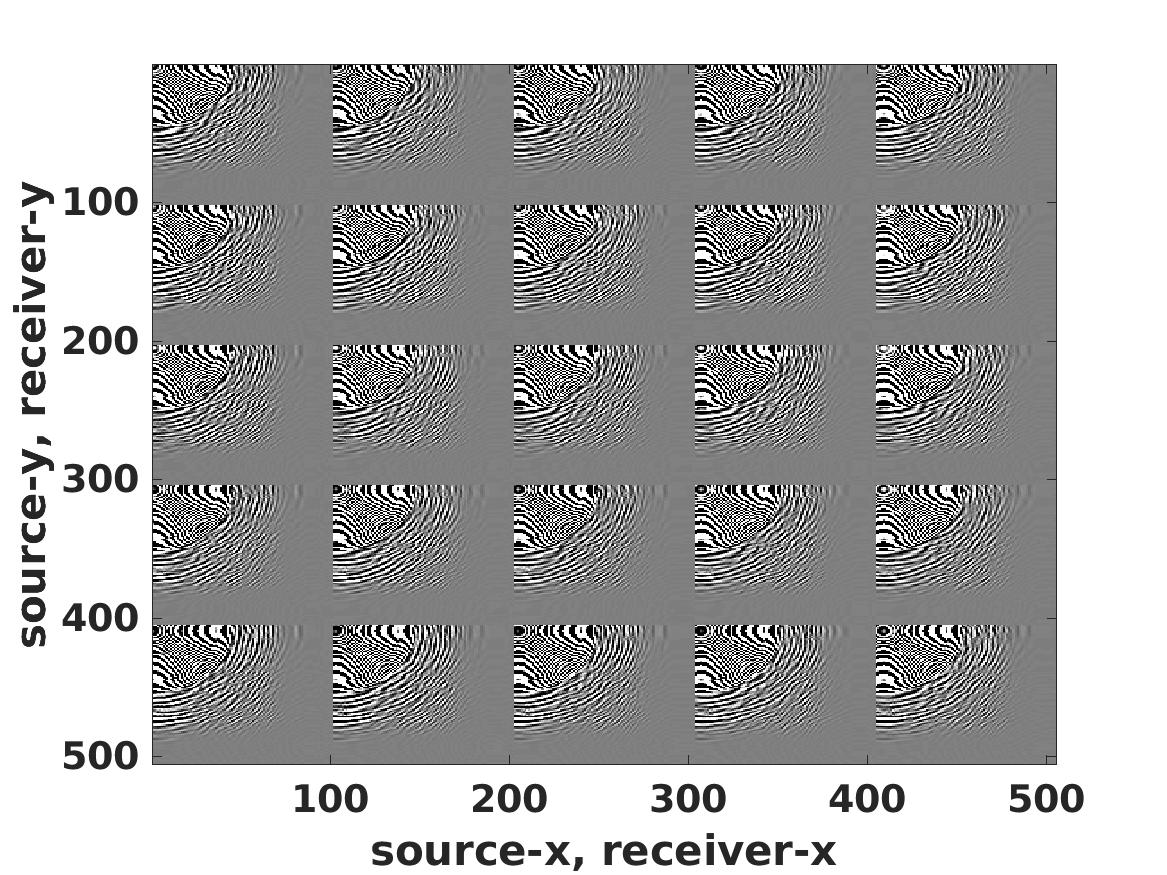}} %\qquad
\subfigure[$l_\infty$]{\label{fig:inter_linf}\includegraphics[scale=0.14]{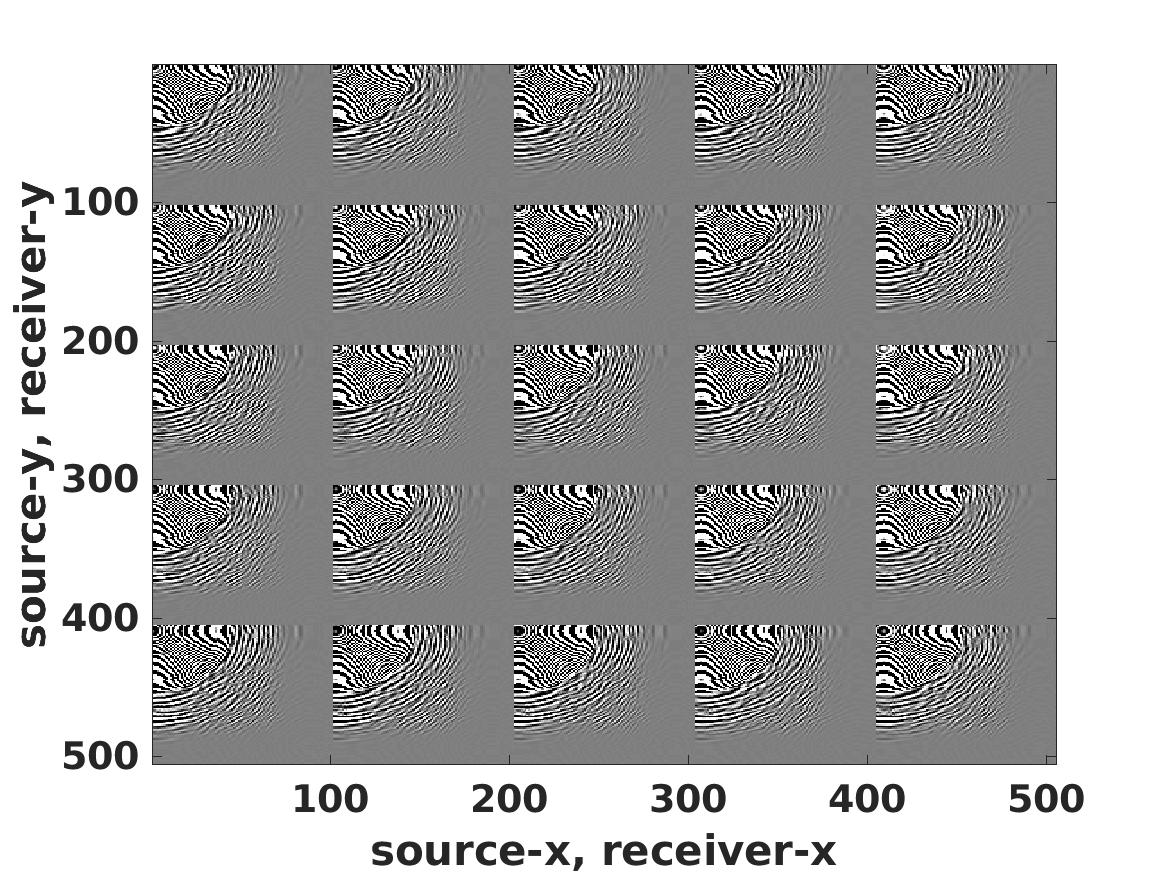}} %\qquad
\subfigure[$l_0$]{\label{fig:inter_l0}\includegraphics[scale=0.14]{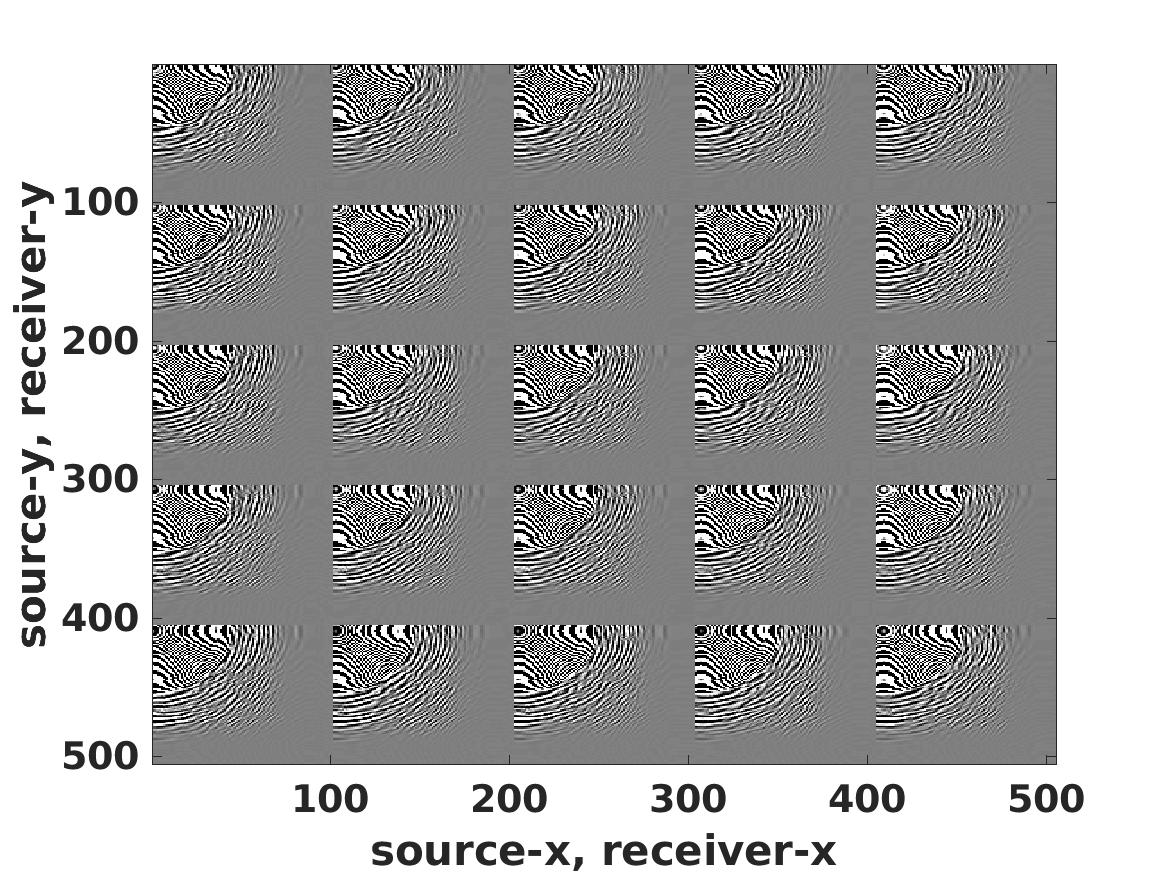}} %\qquad
\caption{Interpolation-only results.}
\label{fig:inter}
\end{figure*}

\begin{figure*}
\centering     %%% not \center
\subfigure[SPGLR]{\label{fig:interdeno_spglr}\includegraphics[scale=0.14]{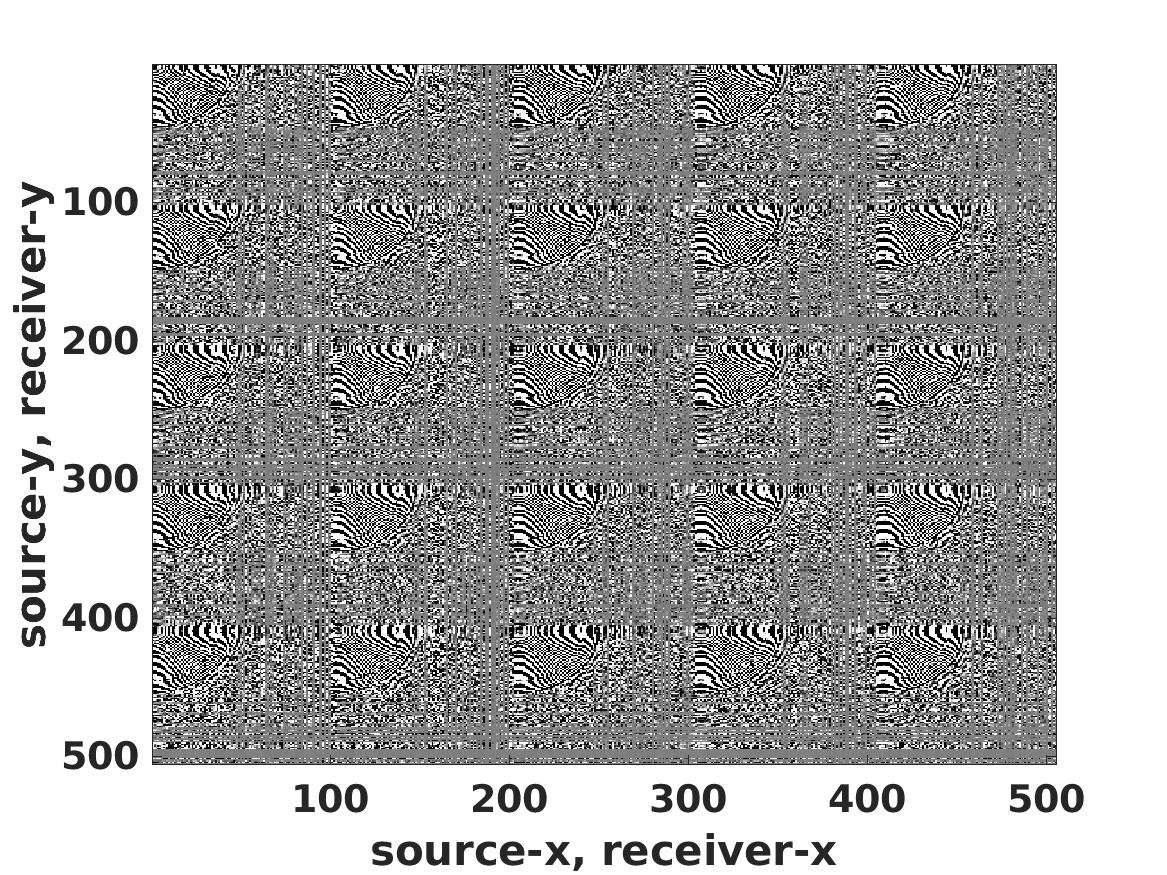}} %\qquad
\subfigure[$l_2$]{\label{fig:interdeno_l2}\includegraphics[scale=0.14]{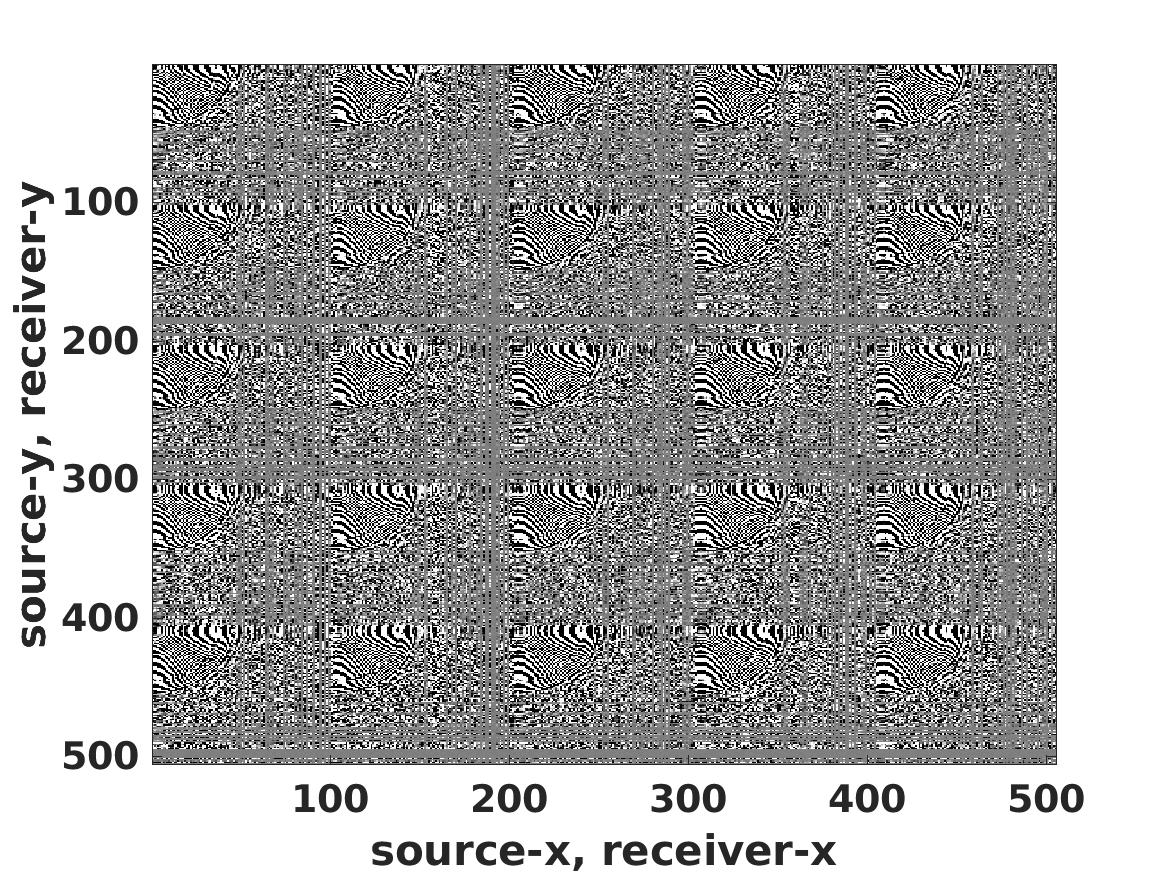}} %\qquad
\subfigure[$l_1$]{\label{fig:interdeno_l1}\includegraphics[scale=0.14]{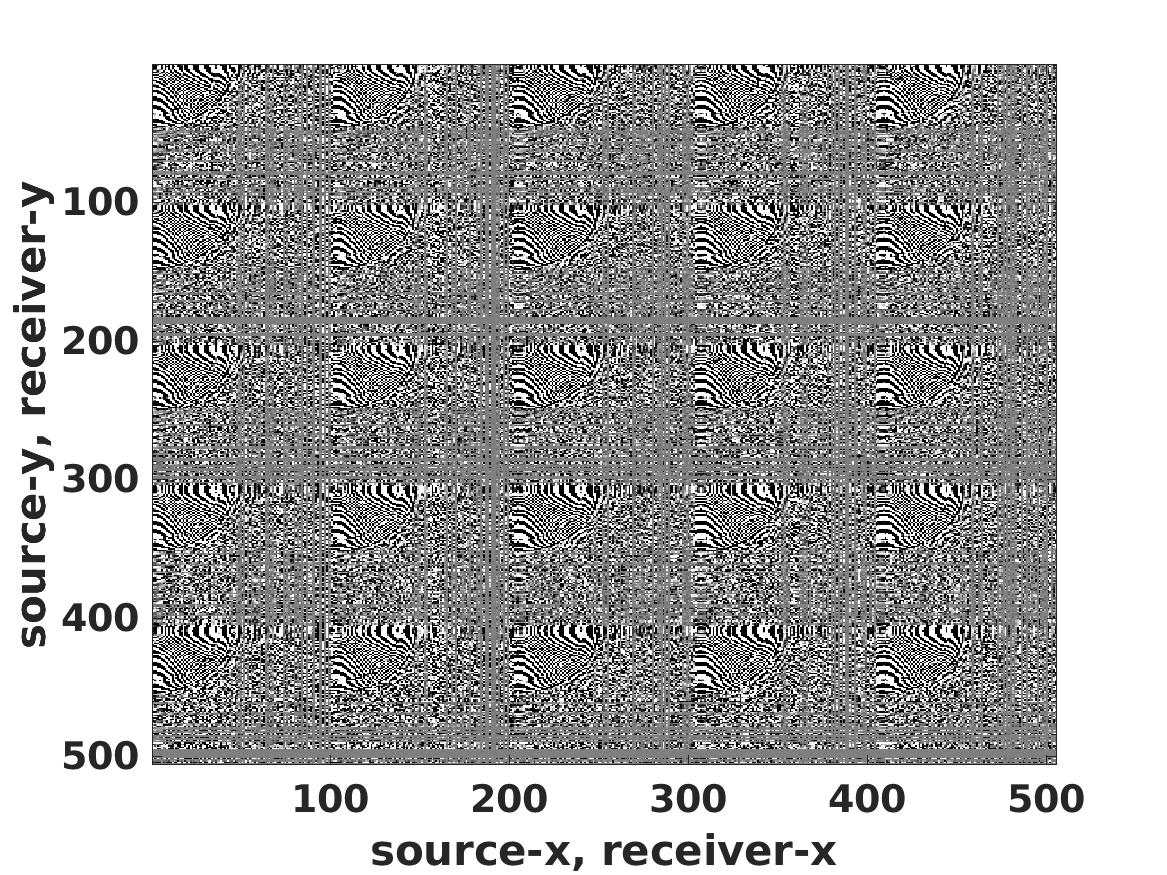}} %\qquad
\subfigure[$l_\infty$]{\label{fig:interdeno_linf}\includegraphics[scale=0.14]{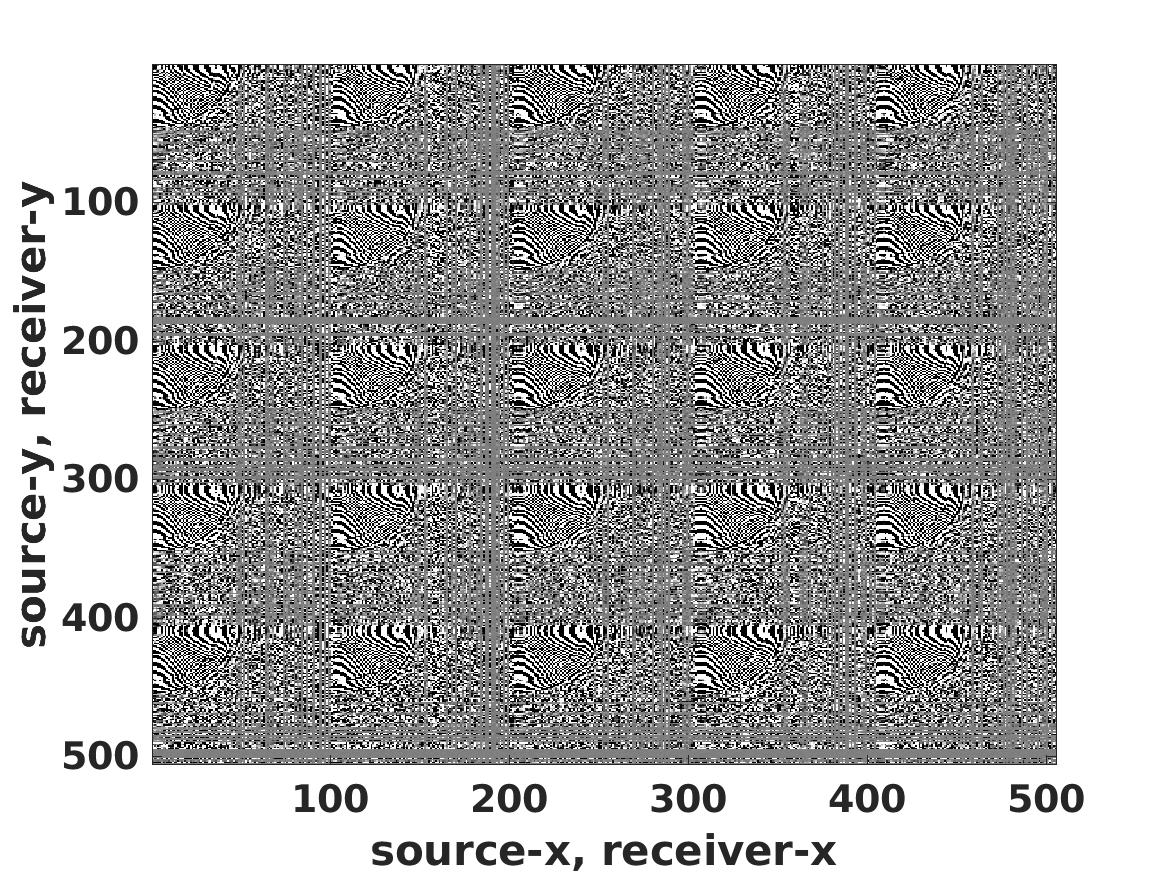}} %\qquad
\subfigure[$l_0$]{\label{fig:interdeno_l0}\includegraphics[scale=0.14]{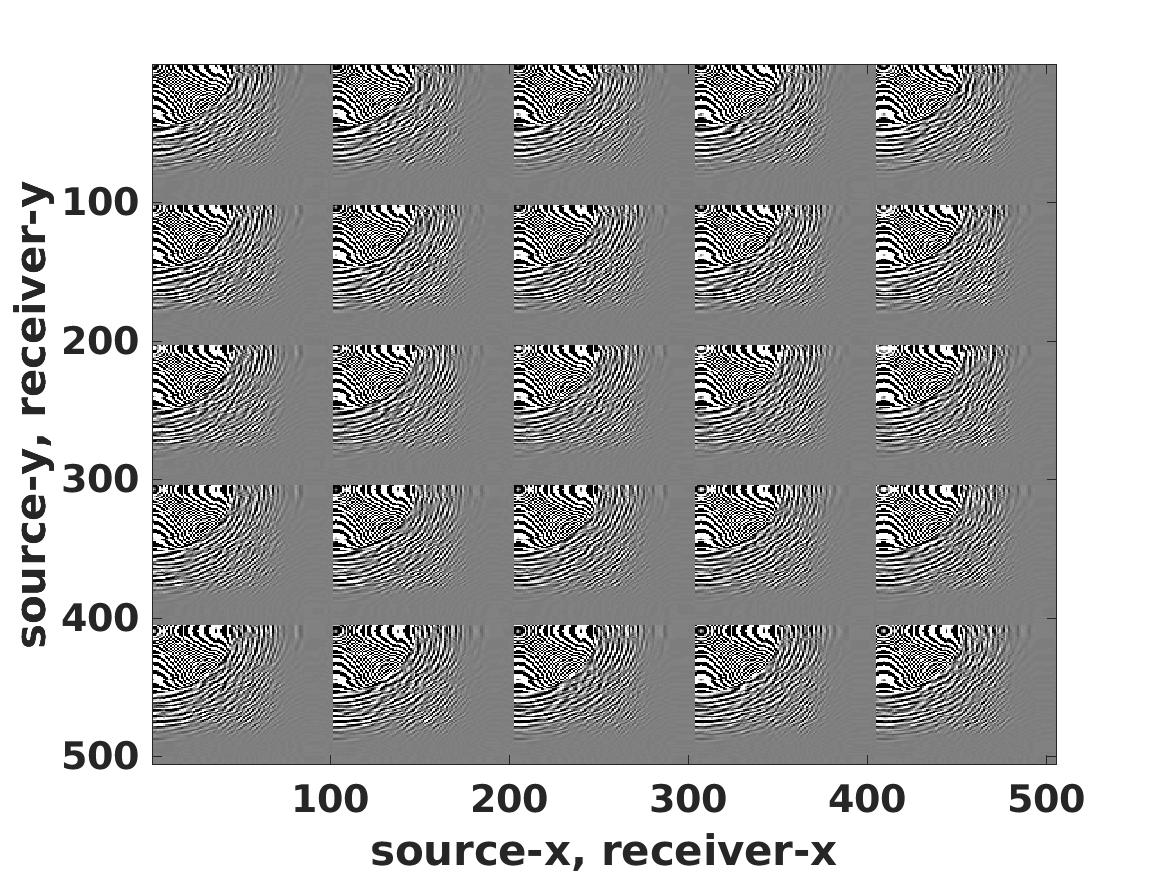}} %\qquad
\caption{Interpolation and Denoising results.}
\label{fig:interdeno}
\end{figure*}

\bibliographystyle{abbrv} 
\bibliography{segbib}

\end{document}